\documentclass[11pt, a4paper%, draft
]{amsart}

\usepackage{amsfonts,graphics, amsthm, amsmath, amscd, amssymb,mathrsfs}
\usepackage[draft=false]{hyperref}
\usepackage[T1]{fontenc}
\usepackage{enumitem}
\usepackage{ae}
\usepackage[all]{xy}
\usepackage{xcolor}
\usepackage{booktabs}
\usepackage{tikz-cd}
\usepackage{tabularx}
\usepackage{booktabs}
\usepackage[title]{appendix}
\usepackage{tikz}
\usetikzlibrary{patterns}
\usepackage{float}

\usepackage{hyperref}

\renewcommand{\le}{\leqslant}
\renewcommand{\ge}{\geqslant}

\numberwithin{equation}{section} % Number equations by section
\newtheorem{theorem}{Theorem}[section] % Theorems numbered within sections
\newtheorem{lemma}{Lemma}[section] % Lemmas numbered within sections
\newtheorem{definition}{Definition}[section] % Definitions numbered within sections
\newtheorem{prop}{Proposition}[section]
\newtheorem{remark}[theorem]{Remark}

\numberwithin{equation}{section}
\theoremstyle{definition}

\renewcommand\rho{\varrho}

\title[Near Isospectrality and Spectral Rigidity for Compact Locally Symmetric Manifolds]{Near Isospectrality and Spectral Rigidity for Compact Locally Symmetric Manifolds}

\subjclass[2020]{58J53, 58C40, 53C24}
\keywords{Spectral rigidity, Locally symmetric spaces, Isospectral manifolds, Heat trace, Curvature invariants}

\author{Sudhir Pujahari}
\address{School of Mathematical Sciences, National Institute of Science Education and Research, Bhubaneswar, An OCC of Homi Bhabha National Institute,  P. O. Jatni,  Khurda 752050, Odisha, India.}
\email{spujahari@niser.ac.in}
\urladdr{https://sites.google.com/site/sudhirkumarpujahari/home}

\author{Punya Plaban Satpathy}
\address{Institute of Mathematics and Applications, Andharua, Bhubaneswar 751029, Odisha, India.}
\email{punya.ima@iomaorissa.ac.in/psatpathy81@gmail.com}

\begin{document}

\begin{abstract}
The inverse spectral problem asks to what extent the Laplace--Beltrami spectrum determines the geometry of a Riemannian manifold. We study a natural weakening, called \emph{near isospectrality}, in which the spectra of two compact manifolds agree outside a finite set, counted with multiplicity. We prove that for compact quotients of a fixed simply connected symmetric space of nonpositive sectional curvature, near isospectrality already forces full isospectrality. We then extend this rigidity to a broad collection of compact quotients of irreducible symmetric spaces of noncompact type. In this larger setting, near isospectrality determines enough heat invariants to identify the universal cover within the class under consideration, and the fixed-cover rigidity result then implies full isospectrality. Thus, within the class studied here, eventual agreement of the Laplace spectrum already forces complete spectral agreement.
\end{abstract}

\maketitle

\bigskip

%\tableofcontents

\section{Introduction and main results} \label{results}

Although the spectrum of the Laplace--Beltrami operator does not determine a compact Riemannian manifold up to isometry, strong forms of spectral rigidity can nevertheless arise in highly structured geometric settings. Classical counterexamples due to Milnor \cite{Milnor}, Vign\'eras \cite{vignerasmarie}, Gordon--Webb--Wolpert \cite{gordonwebb}, and Sunada \cite{sunada} show that isospectral spaces need not be isometric, while simultaneously motivating the search for rigidity phenomena beyond ordinary isospectrality.

In this paper we study a spectral phenomenon for compact locally symmetric manifolds. Let
\[
0=\lambda_0 \le \lambda_1 \le \lambda_2 \le \cdots \uparrow \infty
\]
denote the spectrum of the Laplace--Beltrami operator on a compact Riemannian manifold \(M\), counted with multiplicity. Although the terminology is not standard, for the purposes of this paper we shall say that two compact Riemannian manifolds \(M_1\) and \(M_2\) are \emph{nearly isospectral} if there exists \(N \ge 0\) such that
\[
\lambda_k^{(1)}=\lambda_k^{(2)}
\qquad \text{for all } k\ge N.
\]
Equivalently, the spectra differ in only finitely many eigenvalues. We investigate whether compact locally symmetric manifolds admit nontrivial finite perturbations of the Laplace spectrum. This question may be viewed as a spectral analogue of strong multiplicity one phenomena, in which agreement outside a finite exceptional set forces global equality.

There are two natural settings for this problem. One may first assume that the manifolds share a common universal cover. In this fixed-cover setting, analogues of strong multiplicity one phenomena have been studied extensively using Selberg-type trace formulas and representation-theoretic methods, both of which rely fundamentally on the presence of a fixed ambient Lie group.

The situation becomes considerably more difficult when the universal covers are allowed to vary. One may then ask whether near isospectrality forces the universal covers to coincide. Since the ambient Lie group is no longer fixed, the trace formula and representation-theoretic framework underlying the fixed-cover case no longer directly applies. Instead, we use heat-kernel methods together with the strong rigidity properties of heat invariants on locally symmetric manifolds. This allows us to derive the desired spectral rigidity directly from heat asymptotics, and hence to compare manifolds with different universal covers.

Indeed, eventual spectral agreement implies that the difference of the two heat traces is a finite linear combination of exponential terms, and hence extends real-analytically across \(t=0\). At the same time, the small-time asymptotic expansion of the heat trace encodes geometric information through the heat invariants. For locally symmetric manifolds these invariants are especially rigid: since the curvature tensor is parallel, these invariants are determined entirely by the local isometry type of the universal cover.

Our first result establishes a strong spectral rigidity phenomenon in the
fixed-cover setting: for compact quotients of a fixed simply connected
symmetric space of nonpositive sectional curvature, the Laplace spectrum admits no
nontrivial finite perturbations. In other words, eventual agreement of the
spectrum already forces complete spectral agreement.
\begin{theorem} \label{thm: main1}
 Let $X$ be a simply connected symmetric space of nonpositive sectional curvature. Let $\Gamma_1, \Gamma_2 \subset \operatorname{Isom}(X)$ be torsion-free uniform lattices, and set $M_i := \Gamma_i \backslash X$. Suppose that $M_1$ and $M_2$ are nearly isospectral. Then $M_1$ and $M_2$ are isospectral.
 \end{theorem}

 \begin{remark}
   The proof relies on expressing the heat kernel on each compact quotient in
terms of the minimal positive heat kernel on the common universal cover via the
periodization formula. Gaussian upper bounds then yield exponential decay
of the off-diagonal orbit sums as $t\to0^+$. Since near isospectrality implies
that the difference of heat traces is a finite linear combination of
exponentials, this difference extends real-analytically across $t=0$; the
exponential decay forces the heat traces to agree identically.
 \end{remark}
\newcommand{\col}{\mathscr{C}}

In the varying-cover setting, we consider compact quotients of irreducible symmetric spaces of noncompact type in the classical families. After excluding certain low-dimensional accidental isomorphisms among the underlying symmetric spaces, near isospectrality forces equality of the first nontrivial heat invariants. For locally symmetric manifolds, these become explicit curvature expressions determined by the Lie-theoretic data of the universal cover. 

For the classical families under consideration (see Section~\ref{sec: heat invariant calculation}), these invariants reduce to explicit rational functions of the parameters indexing the symmetric spaces, namely \(n\) or \((p,q)\). Thus the first few heat invariants suffice to distinguish the universal covers within the families considered here.

To state the next result, let \(\col\) denote a fixed set of representatives of the irreducible symmetric spaces of noncompact type belonging to the classical families
\[
\mathrm{AI},\ \mathrm{AII},\ \mathrm{AIII},\ \mathrm{BDI},\ \mathrm{DIII},\ \mathrm{CI},\ \mathrm{CII},
\]
(~\cite[Ch.~X]{helgason1}; see Definition~\ref{the core collection}), each equipped with its standard invariant metric.
Throughout the paper we impose the lower-dimensional restrictions specified in Definition~\ref{the core collection} to exclude accidental isomorphisms and ensure that the heat invariants \(u_1\) and \(u_2\), defined in Definition~\ref{algebraic invariants of symmetric spaces}, distinguish the spaces under consideration.

\begin{theorem}\label{thm: main2}
Let \(X_1, X_2 \in \col\), and assume that \((X_1,X_2)\) does not belong, up to order, to one of the mixed-family configurations
\[
(\mathrm{AIII},\mathrm{BDI}),\qquad
(\mathrm{BDI},\mathrm{CII}),\qquad
(\mathrm{AIII},\mathrm{CII}).
\]
Let \(\Gamma_i \subset \mathrm{Isom}(X_i)\) be torsion-free uniform lattices and \(M_i:=\Gamma_i\backslash X_i\). If \(M_1\) and \(M_2\) are nearly isospectral, then \(X_1=X_2\). Moreover, \(M_1\) and \(M_2\) are isospectral.
\end{theorem}

\begin{remark}
Equality of the invariants leads to Diophantine systems, and we show that, apart from the mixed-family configurations listed in Theorem~\ref{thm: main2}, these systems have no admissible integral solutions. The excluded configurations are exactly the remaining cases not ruled out by the invariant-separation step in Section~\ref{sec: heat invariant calculation}.
\end{remark}

\medskip
These results may be viewed as part of a broader circle of rigidity phenomena closely related to strong multiplicity one principles in number theory and representation theory, where agreement outside a finite set often forces global equality.

In the setting of locally symmetric manifolds, rigidity phenomena of this type have been studied extensively. In the two-dimensional compact hyperbolic case, related rigidity phenomena are well known; see Huber’s work \cite{huber1, Huber2, Huber3} and the expositions in Buser \cite[Ch. 9]{buser} and Chavel \cite[Ch. XI]{Chavel}. Classical work of Elstrodt, Grunewald and Mennicke \cite{Elstrodt} established analogous rigidity results for compact hyperbolic \(3\)-manifolds using analytic trace formula methods. More generally, Bhagwat and Rajan \cite{bhagwatrajanstrongmult} established a spectral version of the strong multiplicity one theorem for uniform lattices in semisimple Lie groups: if all but finitely many irreducible spherical representations occur with the same multiplicities, then the spherical spectra coincide. As a corollary, it follows that compact rank-one locally symmetric manifolds that are nearly isospectral are isospectral. Related representation--spectral converse results, involving representation equivalence and the joint spectrum of invariant differential operators, were studied in \cite{bhagwattau,kelmer,lauretmiatello,Pesce}.

The representation-theoretic results above operate at the level of representation multiplicities, and hence, in the spherical setting, at the level of the full joint spectrum of \(G\)-invariant differential operators; in rank one this reduces to the scalar Laplace spectrum, but in higher rank it is strictly stronger than near isospectrality for the scalar Laplacian. The present work instead allows the ambient symmetric space to vary and relies on heat-kernel methods rather than trace formula techniques.

% In the setting of locally symmetric spaces, such rigidity phenomena have been studied extensively. Classical work of Elstrodt \cite{elstrodt} established analogous rigidity results for compact hyperbolic \(3\)-manifolds, while Bhagwat and Rajan \cite{bhagwatrajanstrongmult} proved a strong multiplicity one theorem for lattices in semisimple Lie groups, implying in particular that compact rank-one locally symmetric spaces which are nearly isospectral are necessarily isospectral. Related representation--spectral converse results appear in \cite{bhagwattau,lauretmiatello,Pesce,kelmer}.

% The representation-spectral converse results of Bhagwat--Rajan and related work operate at the level of representation multiplicities, and hence of the full joint spectrum of \(G\)-invariant differential operators; in rank one this reduces to the scalar Laplace spectrum, but in higher rank it is strictly stronger than near isospectrality for the scalar Laplacian.

% The present work instead allows the ambient symmetric space to vary and relies on heat-kernel methods rather than trace formula techniques.

The paper is organized as follows. In Section~\ref{heat kernel preliminaries}
we recall the necessary geometric and spectral preliminaries. Section~\ref{proof: thm main 1}
contains the proof of Theorem~\ref{thm: main1}. In Section~\ref{sec: heat invariant calculation}
we analyze the heat invariants distinguishing the symmetric spaces in \(\col\)
and in the following section we prove Theorem~\ref{thm: main2}. An appendix records the Lie-theoretic data for the classical families together with a detailed analysis of the Diophantine systems arising in the invariant-separation arguments.

Throughout the paper, \(X\) will denote a simply connected Riemannian manifold.
In Sections~\ref{heat kernel preliminaries} and \ref{proof: thm main 1}, \(X\) is a symmetric space of nonpositive sectional curvature; in later sections we assume that \(X\) is an irreducible symmetric space of noncompact type unless stated otherwise.

\newcommand{\nn}{X}
\section{Preliminaries on heat kernel} \label{heat kernel preliminaries}

In this section, we discuss the heat kernel on simply connected symmetric spaces and their compact Riemannian quotients. For foundational results concerning heat kernels on complete Riemannian manifolds, we refer the reader to Grigor'yan ~\cite{grigoryan} and Li ~\cite{Li2012Geometric}. Throughout the paper, we adopt the analytical convention for the Laplace--Beltrami operator, defining it as a non-negative operator
\[
\Delta = -\operatorname{div}\circ\operatorname{grad},
\]
as in Rosenberg \cite{rosenbergbook}, and we adjust the standard signs in the heat equation accordingly. We use the curvature convention
\[
R(X,Y)Z
=
\nabla_X\nabla_Y Z
-
\nabla_Y\nabla_X Z
-
\nabla_{[X,Y]}Z ,
\]
which differs from the convention in \cite{rosenbergbook} by an overall sign.

\subsection{The Minimal heat kernel on $X$}

Let $(X,\tilde{g})$ be a simply connected symmetric space of nonpositive sectional curvature and dimension $d$. Then $X$ is a Hadamard manifold and its curvature tensor is parallel $(\nabla  R  = 0)$. We define the Laplace--Beltrami operator as the non-negative operator \[ \tilde\Delta := -\operatorname{div}_{\nn} \circ \operatorname{grad}_{\nn}. \]

A central object in analyzing the spectral theory of the Laplacian is the heat kernel.

\begin{definition}\label{heat kernel definition}
A heat kernel on \(X\) is a smooth function
\[
H : X \times X \times (0,\infty) \to \mathbb{R}
\]
such that
\begin{enumerate}
\item \(\displaystyle \frac{\partial H}{\partial t}+\tilde\Delta_x H=0\) on \(X\times X\times(0,\infty)\),
\item for every \(x\in X\) and every \(f\in C_c^\infty(X)\),
\[
\lim_{t\to 0^+}\int_{X} H(x,y,t)f(y)\,d\mathrm{vol}_{\nn}(y)=f(x).
\]
\end{enumerate}
\end{definition}
\begin{remark}
The above definition applies to any complete Riemannian manifold without boundary, both compact and noncompact. In particular, we will use the same definition for compact quotients of $X$.
\end{remark}
% \begin{prop}[{\cite[Theorem 12.4]{Li2012Geometric}}]
% \label{prop:minimal heat kernel}
% There exists a unique minimal positive heat kernel $H(x,y,t)$ on $X$.
% {\color{red} SHOULD WE RESTATE IT?}
% This kernel is smooth and satisfies:
% \begin{enumerate}
% \item Positivity and symmetry: $H(x,y,t) > 0$ and $H(x,y,t) = H(y,x,t)$.
% \item Minimality: If $P(x,y,t)$ is any other non-negative function satisfying
% the heat equation and the initial condition in Definition ~\ref{heat kernel definition},
% then
% $$
% H(x,y,t) \leq P(x,y,t) \ \text{for all $x,y \in X$ and $t > 0$}.
% $$
% \end{enumerate}
% \end{prop}

% \begin{definition}
%     We denote by $H_{\nn}(x,y,t)$ the minimal positive heat kernel on $\nn$ whose existence is established in Proposition ~\ref{prop:minimal heat kernel}.
% \end{definition}

By standard heat kernel theory, there exists a unique minimal positive heat
kernel $H_X(x,y,t)$ on $X$; see (\cite[Theorem 12.4]{Li2012Geometric}). This kernel is smooth, positive,
symmetric, and minimal among all nonnegative fundamental solutions of the
heat equation.

The following result shows that $H_{\nn}(x,y,t)$ only depends on $t$ when $x=y$.

 \begin{lemma} \label{diagonal heat kernel constant}
 
 For any $x \in \nn$, the value of the heat kernel on the diagonal, $H_{\nn}(x, x, t)$, is a function of $t$ alone.\end{lemma}

 \begin{proof}
By \cite[Theorem 9.12]{grigoryan} the heat kernel $H_{\nn}(x,y,t)$ is invariant under isometries: for any $\phi \in \mathrm{Isom}(\nn)$, 
\[
H_{\nn}(\phi(x),\phi(y),t) = H_{\nn}(x,y,t).
\] 
Since $\nn$ is homogeneous, for any $x,y \in \nn$ there exists $\phi \in \mathrm{Isom}(\nn)$ with $\phi(x)=y$, which immediately gives
\[
H_{\nn}(x,x,t) = H_{\nn}(\phi(x),\phi(x),t) = H_{\nn}(y,y,t).
\] 
 \end{proof}

 \begin{lemma} {\normalfont(cf. \cite[Ch. 8, Eq. (7.19)]{grigoryanLMSnotes})} \label{heat kernel estimate} There exists a constant $C > 0$ depending on $d := \operatorname{dim}(\nn)$ and the geometry of $\nn$ such that the minimal positive heat kernel $H_{\nn}(x,y,t)$ satisfies:$$H_{\nn}(x,y,t) \le Ct^{-d/2} \exp\left( -\frac{d_{\nn}(x, y)^2}{6t} \right)$$for all $x, y \in \nn$ and $t \in (0,1)$.\end{lemma}

\subsection{Short-time asymptotics and invariants}

For the remainder of the section we work with a compact quotient of $\nn$. More precisely, let $\Gamma$ be a torsion-free uniform lattice in $\operatorname{Isom}(\nn)$, and define $M := \Gamma \backslash \nn$. Then $M$ is a smooth compact manifold, and its Riemannian structure is given by a metric $g$ inherited from the metric $\tilde{g}$ on $\nn$. Let $\Delta$ be the Laplace-Beltrami operator associated to $(M,g)$.

Since $M$ is compact and has no boundary, $\Delta$ has discrete spectrum
\[
0 = \lambda_0 < \lambda_1 \le \lambda_2 \le \cdots \to \infty,
\]
where $\lambda_0 = 0$ is simple because $M$ is connected. There exists an orthonormal basis of $L^2(M)$ consisting of eigenfunctions $\{\varphi_j\}_{j=0}^\infty$ such that
\[
\Delta \varphi_j = \lambda_j \varphi_j.
\]

Because of the compactness of $M$, there exists a unique heat kernel
$H_M(x,y,t)$ on $M$ satisfying conditions (1) and (2) of
Definition~\ref{heat kernel definition}.

Moreover, $H_M(x,y,t)$ admits the expansion
(~\cite[Proposition~3.1]{rosenbergbook})
\[
H_M(x,y,t)
=
\sum_{j=0}^\infty e^{-\lambda_j t}\varphi_j(x)\varphi_j(y),
\]
which converges smoothly for $t>0$.
% Because of the compactness of $M$, there exists a unique heat kernel $H_M(x,y,t)$ on $M$ satisfying conditions (1) and (2) of Definition ~\ref{heat kernel definition}, and it admits the expansion (cf. \cite[Proposition ~3.1]{rosenbergbook})
% \[
% H_M(x,y,t) = \sum_{j=0}^\infty e^{-\lambda_j t}\varphi_j(x)\varphi_j(y),
% \]
% % % \footnote{On a compact manifold, the heat kernel is the integral kernel of the heat semigroup $e^{-t\Delta}$ on $L^2(M)$. In particular, the initial condition can equivalently be formulated in the $L^2$-sense.}
% which converges smoothly for $t>0$. 
The diagonal values of $H_M$ admit the following short-time asymptotic expansion (see \cite[Proposition ~3.23]{rosenbergbook}):
\begin{equation} \label{short time asymptotics heat kernel}
    H_M(x,x,t) \sim \frac{1}{(4\pi t)^{d/2}} \sum_{j=0}^{\infty} u_j(x,x)t^j \quad \text{as } t \to 0^+.
\end{equation}
Each coefficient $u_j(x,x)$ can be expressed as a universal polynomial in the Riemann curvature tensor of $M$ and its covariant derivatives (~\cite[Lemma 3.26]{rosenbergbook}). By standard properties of the heat kernel, the leading coefficient is $u_0(x,x) = 1$, reflecting the local Euclidean behavior of the manifold.

Recall that the universal cover $X$ of $M$ is a simply connected symmetric space, so its curvature tensor is parallel, $\nabla R = 0$. Since $M=\Gamma\backslash X$, the covering map $X \to M$ is a local isometry, and therefore the curvature tensor of $M$ at $x$ agrees with that of $X$ at any lift $\tilde{x}$. The coefficients $u_j(x,x)$ are universal polynomials in the curvature tensor and its covariant derivatives; hence the coefficient computed at $x \in M$ agrees with the coefficient computed at any lift $\tilde{x} \in X$. Because all covariant derivatives of $R$ vanish on $X$, this common value is independent of $x$ and of the choice of torsion-free uniform lattice $\Gamma$. We denote it by $u_j(X)$. In this sense, the heat invariants are geometric invariants of the universal cover.

 \begin{definition} \label{algebraic invariants of symmetric spaces}
Let $X$ be a simply connected symmetric space of nonpositive sectional curvature, and let
$M = \Gamma \backslash X$ be a compact quotient by a torsion-free uniform lattice.
The diagonal heat kernel on $M$ admits the short-time asymptotic expansion
\begin{equation*}
H_M(x,x,t) \sim  \frac{1}{(4\pi t)^{d/2}} \sum_{j=0}^{\infty} u_j t^j
\quad \text{as } t \to 0^+,
\end{equation*}
where $u_0 = 1$ and the coefficients $u_j$ are independent of $x \in M$
and of the choice of lattice $\Gamma$. We define the \textbf{heat invariants}
of $X$, denoted $u_j(X)$, to be these coefficients $u_j$.
\end{definition}

\begin{remark}
    These invariants are the natural local heat coefficients;
compare Gilkey's normalized coefficients \(\mathcal{A}_n(\mathcal{M})\) for compact locally
symmetric spaces \cite{Gilkeymiatello}.
\end{remark}

\subsection{Heat trace decomposition} By integrating $H_M(x,x,t)$ over $M$, we obtain the trace of the heat semigroup, which directly bridges the spectrum and the geometry of $M$:
\[
Z_M(t):=\operatorname{Tr}(e^{-t\Delta})=\int_M H_M(x,x,t)\,d\mathrm{vol}_M(x)=\sum_{j=0}^\infty e^{-\lambda_j t},
\]
where the eigenvalues are counted with multiplicity. As $t\to 0^+$, integrating the diagonal asymptotic expansion in \eqref{short time asymptotics heat kernel} yields the expansion (see \cite[Theorem 3.24]{rosenbergbook}):
\begin{equation} \label{heat trace asymptotic expansion}
    Z_M(t)\sim \frac{1}{(4\pi t)^{d/2}}\sum_{j=0}^\infty \tilde{u}_j(M)\,t^j \quad \text{as } t \to 0^+,
\end{equation}

where the global heat trace coefficients $\tilde{u}_j(M)$ are given by
\begin{equation} \label{local global heat trace coefficients relation}
    \tilde{u}_j(M)=\operatorname{vol}(M)\,u_j(\nn).
\end{equation}

To prove our first result Theorem ~\ref{thm: main1}, we need to bound the heat trace $Z_M(t)$ for small times. To do this, we first relate the heat kernel on $M$ to the minimal positive heat kernel on $\nn$. More precisely, the heat kernel $H_M(x,y,t)$ on the compact quotient $M$ is obtained from $H_{\nn}(x,y,t)$ by periodization:
\begin{equation} \label{eq: poincare series heat kernel}
    H_M(x,y,t)=\sum_{\gamma\in\Gamma} H_{\nn}(\tilde x,\gamma\tilde y,t),
\end{equation}
where $\tilde{x},\tilde{y}$ are arbitrary lifts of $x,y \in M$. This identity is stated at the beginning of Section~4 of Weber~\cite{Weber2008-kq}. The series in \eqref{eq: poincare series heat kernel} converges absolutely for each $t>0$ and locally uniformly on compact subsets of $\nn\times \nn\times (0,\infty)$; which follow from the Gaussian upper bounds for \(H_{\nn}(x,y,t)\) together with the at most exponential growth of \(\Gamma\)-orbits.

\begin{prop}\label{prop: heat trace decomposition}
Let $\mathcal{F} \subset {\nn}$ be a compact fundamental domain for $\Gamma$. The heat trace $Z_M(t)$ admits the identity
\begin{equation}
    Z_M(t) = \mathrm{Vol}(M)\, H_{\nn}(o, o, t) + Z_M^{\mathrm{off}}(t),
\end{equation}
where $Z_M^{\mathrm{off}}(t) := \displaystyle \int_{\mathcal{F}} \bigg(\sum_{\gamma \in \Gamma \setminus \{\mathrm{id}\}} H_{\nn}(x, \gamma x, t) \bigg)\, d\mathrm{vol}_{\nn}(x)$ and $o \in \nn$ is any fixed basepoint. Moreover, there exists a constant $ L > 0$ such that, as $t \to 0^+$,
\begin{equation}
    Z_M^{\mathrm{off}}(t)
    = O\!\left(t^{-d/2} e^{-L^2/(12t)}\right),
\end{equation}
where $d = \dim \nn$ and $L>0$ is a constant depending only on $\nn$ and $\Gamma$.

\end{prop}

\begin{proof}
Using \eqref{eq: poincare series heat kernel}, we obtain 
\[
Z_M(t)=\int_M H_M(x,x,t)\,d\mathrm{vol}_M(x)
=\int_{\mathcal F}\bigg(\sum_{\gamma\in\Gamma} H_{\nn}(x,\gamma x,t)\bigg)\,d\mathrm{vol}_X(x).
\]
Splitting off the identity term gives
\[
Z_M(t)=\int_{\mathcal F} H_{\nn}(x,x,t)\,d\mathrm{vol}_{\nn}(x)
+\int_{\mathcal F}\bigg(\sum_{\gamma \in \Gamma \setminus \{\mathrm{id}\}} H_{\nn}(x,\gamma x,t) \bigg)\,d\mathrm{vol}_{\nn}(x).
\]

It follows from Lemma ~\ref{diagonal heat kernel constant}  $H_{\nn}(x,x,t)=H_{\nn}(o,o,t)$, hence
\[
\int_{\mathcal F} H_{\nn}(x,x,t)\,d\mathrm{vol}_{\nn}(x)
=\mathrm{Vol}(M)\,H_{\nn}(o,o,t).
\]

From Lemma ~\ref{heat kernel estimate} we have the following estimate for all $x,y \in \nn$:
\[
H_{\nn}(x,y,t)\le C t^{-d/2}\exp\!\left(-\frac{d_{\nn}(x,y)^2}{6t}\right),
\]
with $C> 0 $ as in Lemma ~\ref{heat kernel estimate} and $t \in (0,1)$.  To uniformly bound the sum over $\gamma \in \Gamma \setminus \{\mathrm{id}\}$ in $Z_M(t)$, we split the exponential into two parts. Since $t \in (0,1)$, we have $1/t > 1$, which allows us to write:$$\exp\!\left(-\frac{d_X(x,\gamma x)^2}{6t}\right) \le \exp\!\left(-\frac{d_X(x,\gamma x)^2}{12t}\right)\exp\!\left(-\frac{d_X(x,\gamma x)^2}{12}\right).$$ 
To control the first factor on the right-hand side, we note that the displacement $d_X(x, \gamma x)$ is bounded away from zero on the fundamental domain. Since $\Gamma$ acts freely and properly discontinuously and $M=\Gamma\backslash \nn$ is compact, there exists a compact fundamental domain $\mathcal F \subset \nn$.  Set
\[
L := \inf\{ d_{\nn}(x,\gamma x) : x\in \mathcal{F},\; \gamma\in\Gamma\setminus\{\mathrm{id}\} \}.
\]
We claim $L>0$. Otherwise, there exist $x_n\in \mathcal{F}$ and
$\gamma_n\neq \mathrm{id}$ with $d_{\nn}(x_n,\gamma_n x_n)\to 0$. Passing to a subsequence, $x_n\to x\in \mathcal F$. Then
\[
d_{\nn}(\gamma_n x,x)\le 2\,d_{\nn}(x_n,x)+d_{\nn}(x_n,\gamma_n x_n)\to 0,
\]
so $\gamma_n x\to x$. By proper discontinuity,
$\{\gamma\in\Gamma : \gamma U\cap U\neq\emptyset\}$ is finite for some neighborhood $U$
of $x$, so $\gamma_n$ is eventually constant, say $\gamma_n=\gamma\neq \mathrm{id}$,
forcing $\gamma x=x$; a contradiction. Thus $L>0$.

\noindent Because the integration domain for the off-diagonal sum is restricted to $\mathcal{F}$, we immediately have $d_X(x, \gamma x) \ge L$ for all $x \in \mathcal{F}$ and $\gamma \neq \mathrm{id}$. Substituting this uniform lower bound $L$ into our split exponential yields:
\[
H_{\nn}(x,\gamma x,t)
\le C t^{-d/2}\exp\!\left(-\frac{L^2}{12t}\right)
\exp\!\left(-\frac{d_{\nn}(x,\gamma x)^2}{12}\right).
\]

It remains to bound the purely spatial sum:
\[
\sum_{\gamma \in \Gamma \setminus \{\mathrm{id}\}} \exp\!\left(-\frac{d_{\nn}(x,\gamma x)^2}{12}\right).
\]

Since \(\nn\) is a simply connected symmetric space of nonpositive sectional curvature, it is a Hadamard manifold. As \(\Gamma\) is torsion-free and uniform, it acts freely and cocompactly by isometries. The exponential orbit growth estimate (see Proposition~\ref{lemma:orbit growth}) then yields constants \(\mu,\eta>0\), independent of \(x\), such that
\[
\#\{\gamma\in\Gamma : d_{\nn}(x,\gamma x)<R\}\le \mu e^{\eta R}
\quad \text{for all } x\in \nn,\ R>0.
\]

Writing the sum in shells, we obtain for \(x\in \mathcal F\),

\begin{align*}
    \sum_{\gamma \in \Gamma \setminus \{\mathrm{id}\}} e^{-d_{\nn}(x,\gamma x)^2/12} & \leq \sum_{k=0}^\infty \#\{\gamma\in\Gamma : k \le d_{\nn}(x,\gamma x) < k+1\}\, e^{-k^2/12} \\
    & \leq \mu \sum_{k=0}^\infty e^{\eta(k+1)} e^{-k^2/12}.
\end{align*}
The series $\displaystyle \sum_{k=0}^\infty e^{\eta(k+1)} e^{-k^2/12}$ converges, hence there exists \(K>0\) such that
\[
\sum_{\gamma \in \Gamma \setminus \{\mathrm{id}\}} e^{-d_{\nn}(x,\gamma x)^2/12}\le K
\quad \text{for all } x\in \mathcal F.
\]

Therefore,
\[
\sum_{\gamma \in \Gamma \setminus \{\mathrm{id}\}} H_{\nn}(x,\gamma x,t)
\le C K\, t^{-d/2} \exp\!\left(-\frac{L^2}{12t}\right),
\]
uniformly for $x\in\mathcal F$ and $t\in(0,1)$. Integrating over $\mathcal F$ yields
\[
\int_{\mathcal F}\bigg (\sum_{\gamma \in \Gamma \setminus \{\mathrm{id}\}} H_{\nn}(x,\gamma x,t)\bigg )\,d\mathrm{vol}_{\nn}(x)
\le \mathrm{Vol}(M)\, C K\, t^{-d/2} e^{-L^2/(12t)}.
\]

Hence the remainder term $Z_M^{\mathrm{off}}(t)$ is $O\!\left(t^{-d/2} e^{-L^2/(12t)}\right)$ as $t\to 0^+$, and in particular decays exponentially. 
\end{proof}

% \textcolor{red}{in this case their universal covers are hadamard manifold and symmetric spaces, but not necessarily of noncompact type, they may have flat factors, so figure it out, this may be required to apply the cartan ambrose theorem later.}

% The discussion above reduces the nearly isospectral rigidity problem to
% a comparison of the low-order heat invariants. In the next section, we
% combine the heat trace asymptotics with the fixed covering assumption to
% show that the leading contribution is spectrally rigid up to an
% exponentially decaying error term.

\section{Proof of Theorem ~\ref{thm: main1}} \label{proof: thm main 1}

Since $M_1$ and $M_2$ are nearly isospectral, their Laplace spectra differ by only finitely many eigenvalues. Consequently, their eigenvalue counting functions share the same leading asymptotic behavior. By Weyl's law, this implies that $M_1$ and $M_2$ have the same dimension and the same volume; denote
\[
\dim M_1=\dim M_2=d, \qquad \mathrm{Vol}(M_1)=\mathrm{Vol}(M_2)=V.
\]

Let $Z_{M_i}(t)$ denote the heat trace of $M_i$. By Proposition ~\ref{prop: heat trace decomposition}, we have
\[
Z_{M_i}(t)= V\,H_{\nn}(o,o,t) + Z_{M_i}^{\mathrm{off}}(t),
\]
where
\[
Z_{M_i}^{\mathrm{off}}(t)=O\!\left(t^{-d/2}e^{-L_i^2/(12t)}\right)
\qquad (t\to 0^+),
\]
for some constant $L_i>0$ depending on $\nn$ and $\Gamma_i$.

Setting $L:=\min\{L_1,L_2\}$ and subtracting the two heat traces, we obtain
\begin{equation}\label{eq: heat trace difference clean}
E(t):=Z_{M_1}(t)-Z_{M_2}(t)
=O\!\left(t^{-d/2}e^{-L^2/(12t)}\right)
\qquad (t\to 0^+).
\end{equation}

On the other hand, since $E(t)$ is a finite linear combination of exponentials, we can write
\begin{equation} \label{eq: E(t) clean}
    E(t)=\sum_{i=1}^r c_i e^{-\alpha_i t},
\end{equation}
where the $\alpha_i$ are distinct non-negative real numbers and the $c_i$ are non-zero integers. In particular, $E(t)$ extends to a real-analytic function in a neighborhood of $t=0$. However, the estimate \eqref{eq: heat trace difference clean} implies that
\[
E(t)=o(t^m)\qquad\text{for all }m\ge0. 
\]
Since $E$ is real-analytic in a neighborhood of $0$, all Taylor coefficients of $E$ at $0$ must vanish. Hence
$
E^{(m)}(0)=0$
for all $m\ge0$.
Differentiating \eqref{eq: E(t) clean} 
and then evaluating at $t=0$ yields
\[
E^{(m)}(0)=(-1)^m \sum_{i=1}^r c_i \alpha_i^m.
\]
Thus
\[
\sum_{i=1}^r c_i \alpha_i^m=0 \qquad \text{for } m=0,1,\dots,r-1.
\]

This is a homogeneous linear system in the unknowns $c_1,\dots,c_r$ with coefficient matrix given by the Vandermonde matrix
\[
V=\bigl(\alpha_j^{\,i-1}\bigr)_{1\le i,j\le r}.
\]
Since the $\alpha_i$ are pairwise distinct, the determinant of $V$ is nonzero, and hence $V$ is invertible. Therefore $c_1=\cdots=c_r=0$, contradicting the assumption that mismatched eigenvalues exist.

It follows that no such eigenvalues exist, and hence
\[
Z_{M_1}(t)=Z_{M_2}(t)\quad \text{for all } t>0.
\]
Equality of the heat traces for all $t > 0$ implies equality of the Laplace spectra with multiplicity, as the heat trace is the Laplace transform of the spectral measure. Therefore, $M_1$ and $M_2$ are isospectral.
\qed

\section{Distinguishing the classical families via heat invariants} \label{sec: heat invariant calculation} 

We begin by describing the collection $\col$ in detail, which was outlined in the introduction.

\begin{definition} \label{the core collection}
    Let $\col$ be the collection of irreducible symmetric spaces of noncompact type listed below:
 \begin{center}
\begin{tabular}{c l}
$\mathrm{AI}:$ 
& $\mathrm{SL}(n,\mathbb{R})/\mathrm{SO}(n)$, \quad $n \geq 5$ \\[0.7em]

$\mathrm{AII}:$ 
& $\mathrm{SU}^*(2n)/\mathrm{Sp}(n)$, \quad $n \geq 3$ \\[0.7em]

$\mathrm{AIII}:$ 
& $\mathrm{SU}(p,q)/\mathrm{S}(\mathrm{U}_p \times \mathrm{U}_q)$, \quad $p \geq q \geq 2$ \\[0.7em]

$\mathrm{BDI}:$ 
& $\mathrm{SO}_0(p,q)/(\mathrm{SO}(p)\times\mathrm{SO}(q))$, \quad $p \geq q \geq 5$ \\[0.7em]

$\mathrm{DIII}:$ 
& $\mathrm{SO}^*(2n)/\mathrm{U}(n)$, \quad $n \geq 3$ \\[0.7em]

$\mathrm{CI}:$ 
& $\mathrm{Sp}(n,\mathbb{R})/\mathrm{U}(n)$, \quad $n \geq 3$ \\[0.7em]

$\mathrm{CII}:$ 
& $\mathrm{Sp}(p,q)/(\mathrm{Sp}(p)\times\mathrm{Sp}(q))$, \quad $p \geq q \geq 2$.
\end{tabular}
\end{center}
\end{definition}

\begin{remark}
    These are the standard models for irreducible symmetric spaces of noncompact type arising from the seven classical families in Cartan's classification; see Helgason \cite[Ch. X, \S 6.2, Table V]{helgason1}. For explicit matrix
realizations of the associated classical groups and Lie algebras, we follow
Rossmann \cite[\S 3.1]{rossmann}. In particular, we use the signature
convention \((p,q)\) with \(p\) positive and \(q\) negative entries. We write {\normalfont{\(SO_0(p,q)\)}} for the identity component of \(SO(p,q)\), and the groups
\(SU(p,q)\), \(Sp(p,q)\) are defined with respect to Hermitian forms of signature
\((p,q)\). The parameter lower bounds specified in Definition ~\ref{the core collection} serve two purposes: they exclude the standard exceptional low-dimensional isomorphisms {\normalfont(e.g., $\mathfrak{sl}(2,\mathbb{R}) \cong \mathfrak{su}(1,1) \cong \mathfrak{so}(2,1)$)} and ensure that the invariant-separation arguments used later apply uniformly across the families under consideration.
\end{remark} 

For the spaces in the collection \(\col\), we now study the heat invariants \(\{u_j(X)\}_{j\geq 0}\) introduced in Definition~\ref{algebraic invariants of symmetric spaces}, which are associated to any simply connected symmetric space \(X\) of nonpositive sectional curvature. The goal of this section is to show that, apart from certain mixed-family configurations, the pair \((u_1(X),u_2(X))\) uniquely determines a space \(X\) in the collection $\col$ (see Theorem~\ref{curvature invariant determine spaces}).

\newcommand{\gl}{\mathfrak{g}}
\newcommand{\glp}{\mathfrak{p}}
\newcommand{\glk}{\mathfrak{k}}

\subsection{Lie-theoretic formulas for curvature invariants} \label{subsec: curvature invariant lie calculation}

We will begin by systematically calculating some of the standard curvature invariants for a symmetric space. The curvature invariant calculations in this subsection follow the notation,
conventions, and exposition of Matsushima ~\cite{matsushima}.

Let $X = G/K$ be an irreducible symmetric space of noncompact type (in particular, any space from our collection $\col$), where $G$ is a connected, noncompact, real semisimple Lie group with finite center whose Lie algebra $\mathfrak{g}$ is simple, and $K$ is a maximal compact subgroup. Let $\mathfrak{g} = \mathfrak{k} \oplus \mathfrak{p}$ denote the associated Cartan decomposition of $\mathfrak{g}$, and let $\theta$ be the corresponding Cartan involution. Recall the following commutator relations:
\begin{equation} \label{commutation relation}
[\mathfrak{k},\mathfrak{k}] \subset \mathfrak{k}, \quad [\mathfrak{k},\mathfrak{p}] \subset \mathfrak{p} , \quad [\mathfrak{p},\mathfrak{p}] \subset \mathfrak{k}.
\end{equation}

Let $B_{\gl}(\cdot,\cdot)$ be the Killing form of $\mathfrak{g}$. Because $X$ is of noncompact type, $B_{\gl}$ is positive definite on $\mathfrak{p}$. We use this restriction to define an inner product on $\mathfrak{p} \simeq T_o X$, where $o = eK$ is the fixed base point, and denote this by $\langle \cdot , \cdot \rangle$. We then extend this via the $G$-action to a standard $G$-invariant Riemannian metric $\tilde{g}$ on $X$. 

We will pick a basis $E_1,\dots,E_r$ of $\mathfrak{p}$ and a basis $E_{r+1},\dots,E_n$ of $\mathfrak{k}$ such that:
\begin{equation} \label{killing form pairing}
\begin{split}
    B_{\gl}(E_i, E_j) &= \delta_{ij} \quad \text{for } 1 \leq i, j \leq r, \\
    B_{\gl}(E_{\alpha}, E_{\beta}) &= -\delta_{\alpha\beta} \quad \text{for } r+1 \leq \alpha, \beta \leq n.
\end{split}
\end{equation}

 Throughout this section, we adopt the following indexing conventions: Latin letters ($i, j, k, \dots$) range from $1$ to $r$, while Greek letters ($\alpha, \beta, \gamma, \dots$) denote indices from $r + 1$ to $n$. When the full range from $1$ to $n$ is required, we use the symbols $\lambda, \mu, \nu, \dots$. Let
 \begin{equation}
[E_\lambda, E_\mu] = \sum_{\nu=1}^n c_{\lambda\mu}^\nu E_\nu .
\end{equation}
By \eqref{commutation relation}, among the structure constants $c_{\lambda\mu}^\nu$ only the $c_{\alpha\beta}^\gamma$, $c_{ij}^\alpha$, $c_{j\alpha}^i$, $c_{\alpha j}^i$ can be non zero. We shall write $c_{\alpha ij}$ instead of $c_{ij}^\alpha$.

For any $Z_1,Z_2 \in \glp$, define the endomorphism $R(Z_1,Z_2)$ from $\glp$ to $\glp$ given by

\begin{equation}
    R(Z_1,Z_2)Z = -[[Z_1,Z_2],Z],
\end{equation}
for all $ Z \in \glp$. This endomorphism is the standard algebraic realization of the Riemann curvature tensor of $(X, \tilde{g})$ at the basepoint $o$. Following \cite{matsushima} we set,

\begin{equation}
    R(E_k,E_h)E_j = \sum_{i=1}^r R_{ijkh}E_i.
\end{equation}
Then we get the identity (see \cite[(1.7)]{matsushima}) 
\begin{equation} \label{Riemann curvature structure coefficient relations}
    R_{ijkh} = -\sum_{\alpha=r+1}^n c_{\alpha i j} c_{\alpha kh}.
\end{equation}

We write $R_{jk}$ to denote $\operatorname{Ric}(E_j,E_k)$ for $ 1\leq j , k \leq r$. Following \cite[Ch. 2, \S 3.4]{petersen} it is defined as:
\begin{equation} \label{Ricci tensor definition}
    R_{jk} : = \sum_{i=1}^r \langle R(E_i,E_k)E_j, E_i\rangle.
\end{equation}

Then from \cite[(1.6),(1.7)]{matsushima} we have:
\begin{equation} \label{Ricci tensor formula}
    R_{jk}  = -\frac{1}{2}\delta_{jk}.
\end{equation}
Next we will calculate $|\operatorname{Ric}|^2$, which is defined as:

\begin{equation*}
    |\operatorname{Ric}|^2 : = \sum_{j,k =1}^r R_{jk}R^{jk}. 
\end{equation*}
Using \eqref{Ricci tensor formula} along with the fact that $R^{jk} =R_{jk}$ as $E_1,...,E_r$ is an orthonormal basis of $\glp$ we get,

\begin{equation} \label{Ricci norm squared}
    |\operatorname{Ric}|^2 = \sum_{j,k =1}^r (R_{jk})^2  = 1/4\sum_{j,k =1}^r \delta^2_{jk}  = \frac{\operatorname{dim}(\mathfrak{p})}{4}.
\end{equation}

Next we calculate the scalar curvature which is defined as (see \cite[ Ch. 2, \S 3.5]{petersen}):

\begin{equation} \label{scalar curvature definition}
    \operatorname{scal}: =\operatorname{tr}(\operatorname{Ric}) = \sum _{j=1}^r \sum_{i=1}^r \langle R(E_i,E_j) E_j, E_i \rangle.
\end{equation}
A quick look at \eqref{Ricci tensor definition} shows that the inner sum $\displaystyle \sum_{i=1}^r \langle R(E_i,E_j) E_j, E_i \rangle $ is simply $R_{jj}$. Then from \eqref{Ricci tensor formula} we have:

\begin{equation} \label{scalar curvature formula}
    \operatorname{scal} = \sum _{j=1}^r R_{jj} = -\frac{1}{2}\sum_{j=1}^r \delta_{jj}  = -\frac{\operatorname{dim}(\glp)}{2}. 
\end{equation}

Our final calculation will be for the squared norm of the Riemann curvature tensor, defined as follows:
\begin{equation}
    |\operatorname{Riem}|^2  = \sum_{i,j,k,l=1}^r R_{ijkl}R^{ijkl}.
\end{equation} 
As we are working with the orthonormal basis $E_1,\dots,E_r$ of $\mathfrak{p}$, we have $R^{ijkl}=R_{ijkl}$. Hence
\begin{equation} \label{Riemann norm 1}
    |\operatorname{Riem}|^2 = \sum_{i,j,k,l=1}^r R^2 _{ijkl}.
\end{equation}

To calculate further, we use the identity \eqref{Riemann curvature structure coefficient relations} relating $R_{ijkl}$ to the structure constants of $\mathfrak{g}$. Plugging \eqref{Riemann curvature structure coefficient relations} into \eqref{Riemann norm 1}, we get
\begin{align*}
    |\operatorname{Riem}|^2 & = \sum_{i,j,k,l=1}^r R^2 _{ijkl}  
    = \sum_{\alpha,\beta} \bigg(\sum_{i,j} c_{\alpha i j} c_{\beta ij}  \bigg) \bigg(\sum_{k,l} c_{\alpha kl} c_{\beta kl}  \bigg).
\end{align*}

We proceed as in Matsushima \cite[\S 4]{matsushima} and write the orthogonal decomposition of $\mathfrak{k}$ into its center and simple ideals:
\begin{equation}
    \mathfrak{k} = \mathfrak{z} \oplus \mathfrak{k}_1 \oplus \dots \oplus \mathfrak{k}_q,
\end{equation}
where $\mathfrak{z}$ is the center of $\mathfrak{k}$ and $\mathfrak{k}_1,\dots,\mathfrak{k}_q$ are simple ideals in $\mathfrak{k}$. We now refine our choice of the basis $\{E_{r+1}, \dots, E_n\}$ for $\mathfrak{k}$. While preserving the relation $B_{\gl}(E_{\alpha}, E_{\beta}) = -\delta_{\alpha\beta}$, we additionally require that each $E_{\alpha}$ belongs to the center $\mathfrak{z}$ or to one of the simple ideals $\mathfrak{k}_s$. Under this choice, the cross-terms $\sum_{i,j} c_{\alpha i j} c_{\beta ij}$ vanish for $\alpha \neq \beta$, and the sum collapses to
\begin{equation} \label{eq: riem sq sub final}
    |\operatorname{Riem}|^2 = \sum_{\alpha=r+1}^n  \bigg(\sum_{i,j} c^2_{\alpha i j} \bigg)^2.
\end{equation}

So it is enough to evaluate the sum $\displaystyle \sum_{i,j} c^2_{\alpha i j}$ for $\alpha = r+1$ to $n$. For $Z_1,Z_2 \in \mathfrak{k}$, let
\begin{equation}
    \Psi(Z_1,Z_2)  = \operatorname{tr} (\operatorname{ad}_{\mathfrak{p}} Z_1 \circ \operatorname{ad}_{\mathfrak{p}} Z_2),
\end{equation}
where for $Z \in \glk$, $\operatorname{ad}_{\mathfrak{p}} Z$ is the linear map from $\mathfrak{p}$ to $\mathfrak{p}$ that takes $W \mapsto [Z,W]$ for $W \in \mathfrak{p}$. Then from \cite[(4.2)]{matsushima} we have:
\begin{equation} \label{eq: c sq sum identity}
     \sum_{i,j=1}^r c^2_{\alpha i j} = - \Psi(E_{\alpha},E_{\alpha}) \quad \text{for } \alpha = r+1 \dots n.
\end{equation}

To evaluate $\Psi(E_{\alpha},E_{\alpha})$, we use the ideal decomposition of $\mathfrak{k}$ established above:
\begin{enumerate}
    \item If $E_{\alpha} \in \mathfrak{z}$, then $\Psi(E_{\alpha},E_{\alpha})  = -1$.
    \item If $E_{\alpha} \in \mathfrak{k}_s$ for some $1 \leq s \leq q$, then $\Psi(E_{\alpha},E_{\alpha}) = -a_s$, where $a_s$ is a positive real number such that 
    \begin{equation*}
         \Psi(Z_1,Z_2) = a_s B_{\gl}(Z_1,Z_2) \quad \text{for } Z_1,Z_2 \in \mathfrak{k}_s.
    \end{equation*}
\end{enumerate}

Now we use these values of $\Psi(E_{\alpha},E_{\alpha})$ along with \eqref{eq: riem sq sub final},\eqref{eq: c sq sum identity} to obtain
\begin{align*}
    |\operatorname{Riem}|^2   
    = \sum_{\alpha=r+1}^n  (-\Psi(E_{\alpha},E_{\alpha}))^2  
     = \sum_{s=1}^q \operatorname{dim}(\mathfrak{k}_s)a_s^2 +\operatorname{dim}(\mathfrak{z}).
\end{align*}

For computational purposes, it will be useful to express the constant $a_s$ in terms of another constant $b_s$, which we define below. Let $B_{\mathfrak{k}_s}$ be the Killing form of $\mathfrak{k}_s$. From the decomposition $\mathrm{ad}_{\mathfrak{g}}(Z)=\mathrm{ad}_{\mathfrak{k}}(Z)\oplus \mathrm{ad}_{\mathfrak{p}}(Z)$ for $Z \in \mathfrak{k}$, one obtains the identity
\begin{equation} \label{killing form relations}
    B_{\mathfrak{g}}(Z_1,Z_2)=B_{\mathfrak{k}}(Z_1,Z_2)+\Psi(Z_1,Z_2)
\end{equation}
for $Z_1,Z_2 \in \mathfrak{k}$. The restriction of the Killing form $B_{\mathfrak{k}}$ to $\mathfrak{k}_s$ naturally coincides with its intrinsic Killing form $B_{\mathfrak{k}_s}$, i.e. $B_{\mathfrak{k}}\big|_{\mathfrak{k}_s}=B_{\mathfrak{k}_s}$.

% We now explain why the restriction of $B_{\mathfrak{k}}$ to $\mathfrak{k}_s$ coincides with the Killing form of $\mathfrak{k}_s$. Indeed, if $Z\in\mathfrak{k}_s$, then $\mathfrak{k}_s$ is an ideal of $\mathfrak{k}$, so
% \[
% [\mathfrak{k}_s,\mathfrak{z}]=0, \qquad [\mathfrak{k}_s,\mathfrak{k}_i]=0 \text{ for } i\neq s, \qquad [\mathfrak{k}_s,\mathfrak{k}_s]\subseteq \mathfrak{k}_s.
% \]
% Thus $\operatorname{ad}_{\mathfrak{k}}(Z)$ acts trivially on $\mathfrak{z}$ and on each $\mathfrak{k}_i$ with $i\neq s$, and its only nonzero block is the action on $\mathfrak{k}_s$. In particular, for $Z_1,Z_2\in\mathfrak{k}_s$,
% \[
% \operatorname{ad}_{\mathfrak{k}}(Z_1)\circ \operatorname{ad}_{\mathfrak{k}}(Z_2)
% \]
% has the same trace on $\mathfrak{k}$ as
% \[
% \operatorname{ad}_{\mathfrak{k}_s}(Z_1)\circ \operatorname{ad}_{\mathfrak{k}_s}(Z_2)
% \]
% has on $\mathfrak{k}_s$. Therefore,
% \[
% B_{\mathfrak{k}}\big|_{\mathfrak{k}_s}=B_{\mathfrak{k}_s}.
% \]

Now, we restrict the identity $B_{\mathfrak{g}}=B_{\mathfrak{k}}+\Psi$ to a simple ideal $\mathfrak{k}_s$, and use the fact that any $\mathrm{ad}(\glk_s)$-invariant symmetric bilinear form on $\mathfrak{k}_s$ is a scalar multiple of $B_{\mathfrak{k}_s}$, there exists a constant $b_s$ such that
\[
B_{\mathfrak{k}_s}=b_s\, B_{\mathfrak{g}}|_{\mathfrak{k}_s}.
\]
Substituting into the identity \eqref{killing form relations} yields 
\[
a_s+b_s=1 \implies a_s=1-b_s.
\]
We plug this into our earlier expression for $|\operatorname{Riem}|^2$ and obtain
\begin{equation} \label{Riemann curvature norm final}
|\operatorname{Riem}|^2 = \sum_{s=1}^q \operatorname{dim}(\mathfrak{k}_s)(1-b_s)^2 +\operatorname{dim}(\mathfrak{z}).
\end{equation}

We summarize the curvature invariants calculated in this subsection in the following proposition.

\begin{prop} \label{curvature invariants summary}Let $X = G/K$ be an irreducible symmetric space of noncompact type equipped with the standard $G$-invariant metric $\tilde{g}$ induced by the restriction of the Killing form $B_{\gl}$ to $\mathfrak{p}$. Let $\mathfrak{k} = \mathfrak{z} \oplus \mathfrak{k}_1 \oplus \dots \oplus \mathfrak{k}_q$ be the decomposition of $\mathfrak{k}$ into its center and simple ideals. Then the scalar curvature, the squared norm of the Ricci tensor, and the squared norm of the Riemann curvature tensor of $X$ evaluated at the basepoint $o=eK \in X$ are given by:

\begin{align*}\operatorname{scal} &= -\frac{\operatorname{dim}(\mathfrak{p})}{2}, \ |\operatorname{Ric}|^2 &= \frac{\operatorname{dim}(\mathfrak{p})}{4}, \ |\operatorname{Riem}|^2 &= \sum_{s=1}^q \operatorname{dim}(\mathfrak{k}_s) a_s^2 + \operatorname{dim}(\mathfrak{z}),
\end{align*}
where $a_s : = 1-b_s$ and $b_s$ is the proportionality constant uniquely determined by $B_{\mathfrak{k}_s}=b_s\, B_{\mathfrak{g}}|_{\mathfrak{k}_s}$.\end{prop}

To facilitate the comparison of these invariants across the different families in our collection $\col$, we group the necessary parameters together.  
\begin{definition} \label{def: Cartan decomp structural data}
    For an irreducible symmetric space $G/K$ of noncompact type, we call the quantities,
 \[
\dim \mathfrak p,\quad \dim \mathfrak z,\quad \dim \mathfrak k_s,\quad a_s=1-b_s \ \text{for } 1 \le s \le q,
\] 
where $q$ denotes the number of simple factors in $\mathfrak{k}$ and $b_s$ is defined by $B_{\mathfrak k_s} = b_s\, B_{\mathfrak g}\big|_{\mathfrak k_s}$, as the ``Cartan decomposition and structural data'' for the symmetric space $G/K$.
\end{definition}

\subsection{Uniqueness via the first two nontrivial heat coefficients} 

We now apply the Lie-theoretic calculations of Section~\ref{subsec: curvature invariant lie calculation} to the spaces in our collection $\col$. Recall that in Definition ~\ref{algebraic invariants of symmetric spaces}, we associated a sequence $\{u_j(X)\}_{j \geq 0}$ of heat invariants to any simply connected symmetric space $X$ of nonpositive sectional curvature. Throughout, we use the curvature convention
\[
R(X,Y)Z=\nabla_X\nabla_Y Z-\nabla_Y\nabla_X Z-\nabla_{[X,Y]}Z,
\]
which is the convention fixed in Section~\ref{heat kernel preliminaries}. For $X \in \col$ equipped with the standard metric $\tilde g$ defined in Section ~\ref{subsec: curvature invariant lie calculation}, let $\Delta_X = -\mathrm{div}_{X} \circ \mathrm{grad}_{X}$ denote the associated Laplace--Beltrami operator. The first three invariants are universal scalar invariants (see, for example,
\cite[Chapter~4]{gilkey} and \cite[\S 3.2,3.3]{rosenbergbook}), and hence can be written as
\begin{equation}\label{heat trace coefficients}
\begin{aligned}
u_0(X) &= 1, \\[6pt]
u_1(X) &= a\,\operatorname{scal}, \\
u_2(X) &= b\,|\operatorname{Riem}|^2 + c\,|\operatorname{Ric}|^2 + d\,\operatorname{scal}^2,
\end{aligned}
\end{equation}
where $a,b,c,d$ are universal constants determined by the chosen conventions for $\Delta_X$ and $R$, and are independent of the particular symmetric space $X \in \col$.

% For $X \in \col$ equipped with the standard metric $\tilde{g}$ defined in Section ~\ref{subsec: curvature invariant lie calculation}, let $\Delta_X = -\mathrm{div}_{X} \circ \mathrm{grad}_{X}$ denote the associated Laplace--Beltrami operator. The first three invariants $u_0(X), u_1(X), u_2(X)$ can be expressed entirely in terms of the curvature invariants computed in Proposition ~\ref{curvature invariants summary}. Following \cite[Lemma 3.26, Proposition 3.29 and p. 107]{rosenbergbook}, we have:

% \begin{equation}\label{heat trace coefficients}
% \begin{aligned}
% u_0(X) &= 1, \\[6pt]
% u_1(X) &= \frac{\operatorname{scal}}{6}, \\
% u_2(X) &= \frac{1}{360}\bigl(2|\operatorname{Riem}|^2
% + 2|\operatorname{Ric}|^2 + 5\,\operatorname{scal}^2\bigr).
% \end{aligned}
% \end{equation}

% The first three invariants $u_0(X), u_1(X), u_2(X)$ can be expressed in terms of the curvature invariants computed in Proposition ~\ref{curvature invariants summary}. More precisely, $u_1$ is proportional to the scalar curvature, and $u_2$ is a linear combination of $|\operatorname{Riem}|^2$, $|\operatorname{Ric}|^2$, and $\operatorname{scal}^2$, with universal coefficients (see \cite{rosenbergbook, gilkey}). 

% For symmetric spaces of noncompact type, both $\operatorname{scal}$ and $|\operatorname{Ric}|^2$ are determined by $\dim X$, and hence $u_1$ and $u_2$ are completely determined by $\dim X$ and $|\operatorname{Riem}|^2$.

We now state the main result of this section, which shows that the spaces from the collection $\col$ introduced in Definition ~\ref{the core collection} can be distinguished by the pair $(u_1, u_2)$ except certain mixed family collisions.

\begin{theorem}\label{curvature invariant determine spaces}
Let $\Phi:\col\to \mathbb{R}^2$ be defined by
\[
\Phi(X)=(u_1(X),u_2(X)).
\]
Then $\Phi(X)=\Phi(Y)$ implies $X=Y$, except possibly when, up to interchanging $X$ and $Y$, we have
\[
X \in \mathrm{AIII},\ Y \in \mathrm{BDI};\quad
X \in \mathrm{BDI},\ Y \in \mathrm{CII};\quad
X \in \mathrm{AIII},\ Y \in \mathrm{CII}.
\]
\end{theorem}

\begin{proof}
For a space $X \in \col$, define the invariants
\[
\sigma_1(X) := \dim(X), \qquad \sigma_2(X) := | \mathrm{Riem} |^2,
\]
where $|\mathrm{Riem}|^2$ is the squared norm of the Riemann curvature tensor
with respect to the canonical $G$-invariant metric $\tilde{g}$. Note that $\sigma_i(X)$ for $i=1,2$ can be calculated directly from the Cartan decomposition and structural data of $X$ (see Definition ~\ref{def: Cartan decomp structural data}), with $\sigma_1(X) = \operatorname{dim}(\mathfrak{p})$ and $\sigma_2(X)$ given exactly by Proposition ~\ref{curvature invariants summary}.

Observe that for a space $X \in \col$, it follows from Proposition ~\ref{curvature invariants summary}, \eqref{heat trace coefficients} that the pairs $(u_1(X), u_2(X))$ and $(\sigma_1(X), \sigma_2(X))$ uniquely determine each other, as the terms $\operatorname{scal}$ and $|\operatorname{Ric}|^2$ are simply scalar multiples of $\sigma_1(X)$. Thus, it suffices to prove injectivity for the map $X \longmapsto (\sigma_1(X),\sigma_2(X))$.

We calculate closed form expressions for $\sigma_1,\sigma_2$ in appendix section \ref{appendix} and analyze possible collisions of the pair $(\sigma_1,\sigma_2)$ across the
collection $\col$ in appendix section \ref{sec: collision check}.

First, Proposition ~\ref{prop: collision check 1} shows that no two distinct
spaces among the one-parameter families (AI, AII, DIII, CI) have the same pair
$(\sigma_1,\sigma_2)$. In particular, the map is injective on each of these
families, and no collision occurs between any pair of them.

Next, Propositions ~\ref{prop: collision check 2}, ~\ref{prop: collision check 3},
and ~\ref{prop: collision check 4} rule out all collisions between the
two-parameter families (AIII, BDI, CII) and the one-parameter families. Hence,
if $X$ or $Y$ belongs to a one-parameter family, then
\[
(\sigma_1(X),\sigma_2(X)) = (\sigma_1(Y),\sigma_2(Y)) \implies X = Y.
\]

Finally, Proposition ~\ref{prop: collision check 5} shows that no two distinct
spaces within a fixed two-parameter family have the same pair
$(\sigma_1,\sigma_2)$. It follows that the only remaining possible collisions
are between spaces lying in different two-parameter families, namely among
$\mathrm{AIII}$, $\mathrm{BDI}$, and $\mathrm{CII}$. These mixed-family cases are therefore the only
potential source of non-injectivity.

Combining these results, we conclude that the map
$X \mapsto (\sigma_1(X),\sigma_2(X))$, and hence $\Phi$, is injective on
$\col$ except possibly for cross-family collisions between the three
two-parameter families $\mathrm{AIII}$, $\mathrm{BDI}$, and $\mathrm{CII}$.
Equivalently, if $\Phi(X)=\Phi(Y)$ for $X,Y\in \col$, then $X=Y$ unless
$\{X,Y\}$ is one of the mixed-family configurations
$(\mathrm{AIII},\mathrm{BDI})$, $(\mathrm{BDI},\mathrm{CII})$, or
$(\mathrm{AIII},\mathrm{CII})$, up to order.

\end{proof}

\section{Proof of Theorem ~\ref{thm: main2}}
Since $M_1$ and $M_2$ are nearly isospectral, arguing as in the proof of Theorem~\ref{thm: main1}, we obtain by Weyl's law that both spaces have the same dimension and the same volume. Let $d = \operatorname{dim}(M_i)$ be this common dimension.  For $k=1,2$, denote by
\[
Z_{M_k}(t) = \sum_{j} e^{-\lambda_j^{(k)} t}
\]
the heat trace of $M_k$. The difference of their heat traces can be written as
\[
E(t) := Z_{M_1}(t) - Z_{M_2}(t) = \sum_{i=1}^r c_i e^{-\alpha_i t},
\]
which is a finite linear combination of exponentials. In particular, $E(t)$ is real-analytic at $t=0$ and admits a Taylor expansion involving only nonnegative integer powers of $t$.

On the other hand, each heat trace admits the asymptotic expansion (see \eqref{heat trace asymptotic expansion})
\[
Z_{M_k}(t) \sim  \frac{1}{(4\pi t)^{d/2}} \sum_{j=0}^\infty \tilde{u}_{j}(M_k)\, t^j \quad \text{as } t \to 0^+.
\]
Taking the difference gives
\[
E(t) \sim  \frac{1}{(4\pi t)^{d/2}} \sum_{j=0}^\infty (\tilde{u}_{j}(M_1) - \tilde{u}_{j}(M_2))\, t^j.
\]

Observe that the universal covers of $M_1$ and $M_2$ belong to $\col$. A quick inspection of Definition~\ref{the core collection} reveals that every space in this collection has dimension at least six; hence, the common dimension is $d \ge 6$. Since $E(t)$ is a finite linear combination of exponentials, it extends real-analytically across $t=0$. Therefore its asymptotic expansion as $t\to0^+$ cannot contain any negative powers of $t$. On the other hand, the expansion of $E(t)$ is of the form
\[
 \frac{1}{(4\pi t)^{d/2}}\sum_{j=0}^\infty \bigl(\tilde{u}_j(M_1)-\tilde{u}_j(M_2)\bigr)t^j,
\]
so the terms with $j=0,1,2$ contribute negative powers of $t$. Consequently, the corresponding coefficients must vanish:
\[
\tilde{u}_{j}(M_1) = \tilde{u}_{j}(M_2) \quad \text{for } j =0,1,2.
\]

% Observe that the universal covers of $M_1$ and $M_2$ belong to $\col$. A quick inspection of Definition ~\ref{the core collection} reveals that every space in this collection has dimension at least six; hence, the common dimension is $d \ge 6$. Since \(E(t)\) extends smoothly (indeed, real-analytically) across \(t=0\), its asymptotic expansion as \(t\to0^+\) cannot contain terms with negative powers of \(t\). The powers of $t$ appearing in the expansion of $E(t)$ are of the form $t^{j-d/2}$. Since $d \ge 6$, the first three terms corresponding to $j=0, 1,$ and $2$ yield exponents $j - d/2 \le 2 - 3 = -1$. Consequently, to avoid a singularity at $t=0$, the coefficients of these negative powers must vanish:
% \[
% \tilde{u}_{j}(M_1) = \tilde{u}_{j}(M_2) \quad \text{for  } j =0,1,2.
% \]
As $M_1,M_2$ have the same volume, it follows from \eqref{local global heat trace coefficients relation} that
\begin{equation}
    u_{j}(X_1) = u_{j}(X_2) \quad \text{for  } j =0,1,2.
\end{equation}
It then follows from Theorem ~\ref{curvature invariant determine spaces} that $X_1 = X_2$. Since spaces in $\col$ are simply connected symmetric spaces of nonpositive sectional curvature, we may apply Theorem ~\ref{thm: main1} (setting $X: = X_1 = X_2$) to conclude that $M_1$ and $M_2$ are isospectral.
 \qed

\section{concluding remarks}

The proof of Theorem~\ref{thm: main2} suggests several natural directions for
further investigation. One possibility is to restrict attention to odd-dimensional
locally symmetric manifolds. In this setting, the argument used in the proof of
Theorem~\ref{thm: main2} shows that two nearly isospectral compact locally
symmetric manifolds of nonpositive sectional curvature must have identical global heat trace coefficients of all orders, and
hence the same is true for the heat invariants of their universal covers.

This leads to a natural question concerning the heat invariants introduced in
Definition~\ref{algebraic invariants of symmetric spaces}. To what extent do the
full collection of heat invariants determine a simply connected symmetric space
of nonpositive sectional curvature up to homothety? Equivalently, one may ask how far the
conclusion of Theorem~\ref{curvature invariant determine spaces} can be extended
beyond the class considered in this paper. Since the curvature operator of a
symmetric space encodes the associated restricted root data, this question is
closely related to the extent to which the full system of heat invariants captures
the underlying Lie-theoretic structure of the symmetric space. It is therefore
conceivable that such an invariant-based classification may hold for a substantially
larger class of symmetric spaces of nonpositive sectional curvature.

A different phenomenon arises in even dimensions. In that setting, the heat
kernel method used here detects only finitely many low-order heat invariants
from the spectral tail, and thus the problem of extending
Theorem~\ref{thm: main2} to larger classes of symmetric spaces appears to
require genuinely new input beyond the heat asymptotic framework used in
this paper. Moreover, since the universal cover is not fixed a priori and the
hypothesis involves only eventual agreement of the scalar Laplace spectrum, it
is not clear that standard representation-theoretic or trace formula methods can
be applied directly in this setting.

\section{Acknowledgement}
The authors thank Lizhen Ji and C. S. Rajan for their valuable comments on an earlier version of this paper.
\appendix

\section{Exponential Orbit Growth for Cocompact Actions}

In this appendix we record a standard exponential bound on orbit growth for discrete groups acting cocompactly by isometries on a Hadamard manifold. The result follows from the Švarc--Milnor lemma and elementary growth estimates for finitely generated groups. It is used in the main text in the proof of Proposition \ref{prop: heat trace decomposition}.

\begin{prop} \label{lemma:orbit growth}
Let \(X\) be a Hadamard manifold, and let \(\Gamma \leq \mathrm{Isom}(X)\) be a discrete torsion-free subgroup acting properly discontinuously and cocompactly on \(X\). Then there exist constants \(\mu, \eta > 0\), depending only on \(X\) and \(\Gamma\), such that for every \(x \in X\) and every \(R > 0\),
\[
\#\{ g \in \Gamma \mid d_X(x, gx) < R \} \le \mu e^{\eta R}.
\]
In particular, the constants \(\mu\) and \(\eta\) can be chosen independently of the basepoint \(x\).
\end{prop}

\begin{proof}
Fix \(x_0 \in X\). By the Švarc--Milnor lemma \cite[Prop.~8.19]{bridson_haefliger}, the group \(\Gamma\) is finitely generated and the orbit map \(g \mapsto g x_0\) is a quasi-isometry. Hence, for some constants \(A \ge 1\) and \(B \ge 0\),
\[
A^{-1}|g|_S - B \le d_X(x_0,gx_0) \le A|g|_S + B
\qquad (g \in \Gamma),
\]
where \(|\cdot|_S\) denotes the word metric with respect to a finite generating set \(S\) of \(\Gamma\).

Therefore,
\[
d_X(x_0,gx_0) < R
\implies
|g|_S < A(R+B),
\]
and hence
\[
\{g \in \Gamma : d_X(x_0,gx_0) < R\}
\subseteq
\{g \in \Gamma : |g|_S \le \lfloor A(R+B)\rfloor\}.
\]
Since \(S\) is finite, there exists \(C_1,\eta_1>0\) such that
\[
\#\{g \in \Gamma : |g|_S \le n\} \le C_1 e^{\eta_1 n}
\qquad (n \ge 0).
\]
Applying this with \(n=\lfloor A(R+B)\rfloor\), we obtain
\[
\#\{g \in \Gamma : d_X(x_0,gx_0) < R\} \le \mu_0 e^{\eta_0 R}
\]
for suitable constants \(\mu_0,\eta_0>0\).

Now let \(x \in X\). Since \(\Gamma\) acts cocompactly, there exists \(0<D<\infty\) such that for every \(x\in X\) one can find \(\gamma\in\Gamma\) with
$d_X(x,\gamma x_0)\le D$. Then for any \(g\in \Gamma\),
\[
d_X(\gamma x_0,g\gamma x_0)
\le d_X(\gamma x_0,x)+d_X(x,gx)+d_X(gx,g\gamma x_0)
\le d_X(x,gx)+2D.
\]
Thus
\[
\{g \in \Gamma : d_X(x,gx) < R\}
\subseteq
\{g \in \Gamma : d_X(\gamma x_0,g\gamma x_0) < R+2D\}.
\]
Using \(\Gamma\)-invariance of \(d_X\),
\[
d_X(\gamma x_0,g\gamma x_0)=d_X(x_0,\gamma^{-1}g\gamma x_0),
\]
and conjugation by \(\gamma\) is a bijection of \(\Gamma\). Hence
\[
\#\{g \in \Gamma : d_X(x,gx) < R\}
\le
\#\{h \in \Gamma : d_X(x_0,hx_0) < R+2D\}
\le \mu e^{\eta R}
\]
for suitable constants \(\mu,\eta>0\) independent of \(x\).
\end{proof}

\section{Cartan decomposition and structural data for the collection $\col$} \label{appendix}

 In this section, we detail the Cartan decomposition and structural data (see Definition ~\ref{def: Cartan decomp structural data}) for the collection $\col$ of irreducible symmetric spaces of noncompact type. This data provides the explicit parameters required to compute the geometric invariants $\sigma_1(X)$ and $\sigma_2(X)$ introduced in the proof of Theorem ~\ref{curvature invariant determine spaces}. 

The numerical entries below are obtained from the standard classification
data and Killing-form formulas for the classical Lie algebras
(see \cite{helgason1}) together with the matrix realizations and signature
conventions of Rossmann \cite[\S 3.1]{rossmann}.

The values of \(\dim \mathfrak p\), \(\dim \mathfrak z\), the number of simple
factors \(q\) of $\mathfrak{k}$, and the dimensions of the \(\mathfrak k_i\) are standard and may be
read off from the classification of irreducible symmetric spaces; see, for example,
Helgason \cite[Chapter 3, \S 8; Chapter 10, \S 6.2, Table V]{helgason1}.

Let \(\mathfrak g\) be a real simple Lie algebra of classical type, realized in its standard matrix model (see \cite[\S 3.1]{rossmann}). Then its Killing form has the form
\[
B_{\mathfrak g}(X,Y)=c_{\mathfrak g}\,\operatorname{Re}\operatorname{Tr}(XY),
\]
where the constant \(c_{\mathfrak g}\) depends only on the complexification \(\mathfrak g_{\mathbb C}\) (see \cite[Ch. III, \S 8]{helgason1}) and is given by
\[
\begin{cases}
c_{\mathfrak g} = 2n, & \text{if } \mathfrak{g}_{\mathbb{C}} \cong \mathfrak{sl}(n,\mathbb C),\\[2mm]
c_{\mathfrak g} = N-2, & \text{if } \mathfrak{g}_{\mathbb{C}} \cong \mathfrak{so}(N,\mathbb C),\\[2mm]
c_{\mathfrak g} = 2(n+1), & \text{if } \mathfrak{g}_{\mathbb{C}} \cong \mathfrak{sp}(n,\mathbb C).
\end{cases}
\]
In the cases considered here, \(\operatorname{Tr}(XY)\) is real-valued on \(\mathfrak g\), so
\[
B_{\mathfrak g}(X,Y)=c_{\mathfrak g}\,\operatorname{Tr}(XY).
\]
This follows from the fact that for each classical complex simple Lie algebra \(\mathfrak s\) in its defining matrix realization, there exists \(c\neq 0\) such that
\[
B(X,Y)=c\,\operatorname{Tr}(XY)\quad\text{for all }X,Y\in\mathfrak{s}
\]
(see \cite[\S 1.18, Exercise 13]{knapp}), together with the compatibility of Killing forms under complexification (see \cite[Lemma 6.1]{helgason1}).

\medskip

For each simple factor \(\mathfrak k_s \subset \mathfrak k\), the restriction
\(B_{\mathfrak g}\big|_{\mathfrak k_s}\) is an invariant symmetric bilinear form on
\(\mathfrak k_s\). Since \(\mathfrak k_s\) is simple, it is proportional to its intrinsic Killing form \(B_{\mathfrak k_s}\), so there exists a constant \(b_s\) such that
\[
B_{\mathfrak k_s} = b_s\, B_{\mathfrak g}\big|_{\mathfrak k_s}.
\]
Using the standard matrix realizations of the classical real forms and the corresponding embeddings of the compact factors, both Killing forms are given by trace formulas with constants \(c_{\mathfrak g}\) and \(c_{\mathfrak k_s}\). Comparing these trace formulas yields the coefficient $b_s$, and hence
\[
a_s = 1-b_s.
\]
For instance, consider Type AIII, where $\mathfrak{g} = \mathfrak{su}(p,q)$. The complexification is $\mathfrak{sl}(p+q, \mathbb{C})$, yielding $c_{\mathfrak{g}} = 2(p+q)$. The maximal compact subgroup $K = S(U_p \times U_q)$ has Lie algebra $\mathfrak{k} = \mathfrak{s}(\mathfrak{u}_p \oplus \mathfrak{u}_q)$, which contains the simple factors $\mathfrak{k}_1 = \mathfrak{su}(p)$ and $\mathfrak{k}_2 = \mathfrak{su}(q)$. Focusing on the first factor, its complexification is $\mathfrak{sl}(p, \mathbb{C})$, yielding $c_{\mathfrak{k}_1} = 2p$. It immediately follows that $b_1 = 2p/2(p+q) = p/(p+q)$, which gives $a_1 = 1 - b_1 = q/(p+q)$.

Applying this uniform methodology yields the \textit{Cartan decomposition and structural data} for each family in $\col$:
 \begin{enumerate}[label=\textbf{(\Roman*)}]
\item \textbf{Type AI:} $\mathrm{SL}(n,\mathbb{R})/\mathrm{SO}(n)$, $n \ge 5$.
\[
\dim\mathfrak{p} = \frac{n(n+1)}{2}-1,\quad
\dim\mathfrak{z}=0,\quad
\dim\mathfrak{k}_1 = \frac{n(n-1)}{2},\quad
a_1 = \frac{n+2}{2n}.
\]

\item \textbf{Type AII:} $\mathrm{SU}^*(2n)/\mathrm{Sp}(n)$, $n \ge 3$.
\[
\dim\mathfrak{p} = 2n^2 - n - 1,\quad
\dim\mathfrak{z}=0,\quad
\dim\mathfrak{k}_1 = n(2n+1),\quad
a_1 = \frac{n-1}{2n}.
\]

\item \textbf{Type AIII:} $\mathrm{SU}(p,q)/\mathrm{S}(\mathrm{U}_p \times \mathrm{U}_q)$, $p \ge q \ge 2$.
\[
\dim\mathfrak{p} = 2pq,\quad
\dim\mathfrak{z}=1,\quad
\dim\mathfrak{k}_1 = p^2-1,\quad
\dim\mathfrak{k}_2 = q^2-1,
\]
\[
a_1 = \frac{q}{p+q},\quad
a_2 = \frac{p}{p+q}.
\]

\item \textbf{Type BDI:} $\mathrm{SO}_0(p,q)/(\mathrm{SO}(p)\times\mathrm{SO}(q))$, $p \ge q \ge 5$.
\[
\dim\mathfrak{p} = pq,\quad
\dim\mathfrak{z}=0,\quad
\dim\mathfrak{k}_1 = \frac{p(p-1)}{2},\quad
\dim\mathfrak{k}_2 = \frac{q(q-1)}{2},
\]
\[
a_1 = \frac{q}{p+q-2},\quad
a_2 = \frac{p}{p+q-2}.
\]

\item \textbf{Type DIII:} $\mathrm{SO}^*(2n)/\mathrm{U}(n)$, $n \ge 3$.
\[
\dim\mathfrak{p} = n(n-1),\quad
\dim\mathfrak{z}=1,\quad
\dim\mathfrak{k}_1 = n^2-1,\quad
a_1 = \frac{n-2}{2(n-1)}.
\]

\item \textbf{Type CI:} $\mathrm{Sp}(n,\mathbb{R})/\mathrm{U}(n)$, $n \ge 3$.
\[
\dim\mathfrak{p} = n(n+1),\quad
\dim\mathfrak{z}=1,\quad
\dim\mathfrak{k}_1 = n^2-1,\quad
a_1 = \frac{n+2}{2(n+1)}.
\]

\item \textbf{Type CII:} $\mathrm{Sp}(p,q)/(\mathrm{Sp}(p)\times\mathrm{Sp}(q))$, $p \ge q \ge 2$.
\[
\dim\mathfrak{p} = 4pq,\quad
\dim\mathfrak{z}=0,\quad
\dim\mathfrak{k}_1 = p(2p+1),\quad
\dim\mathfrak{k}_2 = q(2q+1),
\]
\[
a_1 = \frac{q}{p+q+1},\quad
a_2 = \frac{p}{p+q+1}.
\]
\end{enumerate}

\smallskip
\noindent  

Having established the explicit Lie-theoretic parameters for each family, we now evaluate the geometric invariants $\sigma_1(X) = \dim(\mathfrak{p})$ and $\sigma_2(X) = |\operatorname{Riem}|^2$ introduced in the proof of Theorem ~\ref{curvature invariant determine spaces}. Substituting the structural data listed above into the formulas derived in Proposition ~\ref{curvature invariants summary} yields the following closed-form expressions for $\sigma_1,\sigma_2$ for the spaces in $\col$:

\begin{table}[htbp]
\centering
\small
\begin{tabularx}{\textwidth}{
  >{\raggedright\arraybackslash}X 
  c 
  c 
  >{\centering\arraybackslash}X
}
\toprule
\textbf{The Space \(X\)} & \textbf{Type} & \(\boldsymbol{\sigma_1(X)}\) & \(\boldsymbol{\sigma_2(X)}\) \\
\midrule
\(\mathrm{SL}(n,\mathbb{R})/\mathrm{SO}(n)\) (\(n \ge 5\)) 
& AI 
& \(\frac{(n-1)(n+2)}{2}\) 
& \(\frac{(n-1)(n+2)^2}{8n}\) \\[6pt]

\(\mathrm{SU}^*(2n)/\mathrm{Sp}(n)\) (\(n \ge 3\)) 
& AII 
& \((2n+1)(n-1)\)
& \(\frac{(2n+1)(n-1)^2}{4n}\) \\[6pt]

\(\mathrm{SU}(p,q)/\mathrm{S}(\mathrm{U}_p \times \mathrm{U}_q)\) (\(p \ge q \ge 2\)) 
& AIII 
& \(2pq\)
& \(\frac{pq(2pq+2)}{(p+q)^2}\) \\[6pt]

\(\mathrm{SO}_o(p,q)/(\mathrm{SO}(p)\times\mathrm{SO}(q))\) (\(p \ge q \ge 5\))
& BDI 
& \(pq\)
& \(\frac{pq(2pq-(p+q))}{2(p+q-2)^2}\) \\[6pt]

\(\mathrm{SO}^*(2n)/\mathrm{U}(n)\) (\(n \ge 3\))
& DIII 
& \(n(n-1)\)
& \(\frac{n(n^2 - 3n + 4)}{4(n-1)}\) \\[6pt]

\(\mathrm{SP}(n,\mathbb{R})/\mathrm{U}(n)\) (\(n \ge 3\))
& CI 
& \(n(n+1)\)
& \(\frac{n(n^2 + 3n + 4)}{4(n+1)}\) \\[6pt]

\(\mathrm{SP}(p,q)/(\mathrm{SP}(p)\times\mathrm{SP}(q))\) (\(p \ge q \ge 2\))
& CII 
& \(4pq\)
& \(\frac{pq(4pq+p+q)}{(p+q+1)^2}\) \\
\bottomrule
\end{tabularx}
\caption{Closed-form expressions for the geometric invariants $\sigma_1$ and $\sigma_2$ across the collection $\col$.}
\label{tab:symmetric-spaces-sigma invariant}
\end{table}

\section{Verification of Invariant Injectivity} \label{sec: collision check}

In this section, we verify the injectivity of the invariant map $\Phi$ on the collection
$\col$. By the discussion in the proof of Theorem ~\ref{curvature invariant determine spaces},
it suffices to rule out collisions of the pair $(\sigma_1,\sigma_2)$ in all cases except the
potential cross-family collisions among $\mathrm{AIII}$, $\mathrm{BDI}$, and $\mathrm{CII}$. The remaining cases are
treated by a case-by-case analysis of the associated Diophantine equations.

For the following verifications, we refer to the closed-form expressions for $\sigma_1$ and $\sigma_2$ established in Table ~\ref{tab:symmetric-spaces-sigma invariant}. We start by showing that there is no collision among the four one-parameter families.

\begin{prop} \label{prop: collision check 1}
No two distinct spaces among the families \emph{AI}, \emph{AII}, \emph{DIII}, and \emph{CI} have the same pair \((\sigma_1,\sigma_2)\).
\end{prop}

\begin{proof}
Note that within each one-parameter family, collisions cannot occur,
since \(\sigma_1(X)=\dim(X)\) uniquely determines the parameter.
To analyze possible collisions between different families, we introduce,
for \(X \in \col\), the normalized invariant ratio
\[
\rho(X):=4\frac{\sigma_2(X)}{\sigma_1(X)}.
\]
Using the expressions from Table~\ref{tab:symmetric-spaces-sigma invariant},
we compute this ratio for the four one-parameter families:

\begin{enumerate}
    \item $\mathrm{AI}$ ($n \ge 5$): $\rho = \frac{n+2}{n} = 1 + \frac{2}{n} > 1$.
    \item $\mathrm{AII}$ ($n \ge 3$): $\rho = \frac{n-1}{n} = 1 - \frac{1}{n} < 1$.
    \item $\mathrm{DIII}$ ($n \ge 3$): $\rho = \frac{n^2-3n+4}{(n-1)^2} = 1 + \frac{3-n}{(n-1)^2} \le 1$.
    \item $\mathrm{CI}$ ($n \ge 3$): $\rho = \frac{n^2+3n+4}{(n+1)^2} = 1 + \frac{n+3}{(n+1)^2} > 1$.
\end{enumerate}
This strict partitioning instantly rules out any collision between the families $\{ \text{AI}, \text{CI} \}$ and the families $\{ \text{AII}, \text{DIII} \}$. It remains only to check for collisions within each respective partition.
\begin{enumerate}

\item[\textbf{AI vs CI:}]
Eliminating from \(\sigma_1^{\mathrm{AI}}(n)=\sigma_1^{\mathrm{CI}}(m)\) and
\(\sigma_2^{\mathrm{AI}}(n)=\sigma_2^{\mathrm{CI}}(m)\) gives
\[
\frac{n+2}{n}=\frac{m^2+3m+4}{(m+1)^2} \implies n=\frac{2(m+1)^2}{m+3}.
\]
Substituting into \(\sigma_1^{\mathrm{AI}}(n)=\sigma_1^{\mathrm{CI}}(m)\) yields \\
$
m^4+2m^3+m^2-4=(m-1)(m+2)(m^2+m+2)=0,
$
which has no solution with \(m\ge 3\).
\smallskip
\item[\textbf{AII vs DIII:}]
Eliminating from \(\sigma_1^{\mathrm{AII}}(n)=\sigma_1^{\mathrm{DIII}}(m)\) and
\(\sigma_2^{\mathrm{AII}}(n)=\sigma_2^{\mathrm{DIII}}(m)\) gives
\[
\frac{n-1}{n}=\frac{m^2-3m+4}{(m-1)^2}.
\]
If \(m\neq 3\), then
\[
n=\frac{(m-1)^2}{m-3}.
\]
Substituting into \(\sigma_1^{\mathrm{AII}}(n)=\sigma_1^{\mathrm{DIII}}(m)\) yields
\[
(m-2)(m+1)(m^2-m+2)=0.
\]
So the only positive integer root is \(m=2\), which is not allowed. For \(m=3\), the first equality gives
$
2n^2-n-7=0,
$ which has no integral solution.
\smallskip

\end{enumerate}
Thus no two distinct spaces among the four one-parameter families have the
same pair $(\sigma_1,\sigma_2)$.
\end{proof}

We now check that $\mathrm{AIII}$ does not collide with AI, AII, DIII, or CI.
\begin{prop} \label{prop: collision check 2}
Let \(X^{\mathrm{AIII}}_{p,q}=\mathrm{SU}(p,q)/\mathrm{S}(\mathrm{U}_p\times \mathrm{U}_q)\) with \(p\ge q\ge 2\), and let \(Y\) be any space in one of the families $\mathrm{AI}, \mathrm{AII}, \mathrm{DIII}$, or $\mathrm{CI}$, with admissible parameter $n$. Then
\[
\bigl(\sigma_1(X^{\mathrm{AIII}}_{p,q}),\,\sigma_2(X^{\mathrm{AIII}}_{p,q})\bigr)
\neq
\bigl(\sigma_1(Y),\,\sigma_2(Y)\bigr).
\]
\end{prop}

\begin{proof}
Write
\[
s:=p+q.
\]
Since
\[
\sigma_1^{\mathrm{AIII}}(p,q)=2pq,
\qquad
\sigma_2^{\mathrm{AIII}}(p,q)=\frac{pq(2pq+2)}{s^2},
\]
any potential collision with a one-parameter family \(Y\) forces their invariants to coincide. Setting $M = \sigma_1(Y)$, this gives us
\[
2pq = M \qquad \text{from which we get} \qquad \sigma_2^{\mathrm{AIII}}(p,q) = \frac{M(M+2)}{2s^2}.
\]

Then  equating the \(\sigma_2\)-invariants yields an explicit expression for \(s^2\). In each case we obtain:

\[
\renewcommand{\arraystretch}{1.6}
\begin{array}{c|c|c}
\text{Family }Y & M=\sigma_1(Y) & s^2 \\ \hline
\mathrm{AI}(n) &
\dfrac{(n-1)(n+2)}{2} &

n^2-n+4-\dfrac{8}{n+2}
\\
\mathrm{AII}(n) &
(2n+1)(n-1) &

4n^2+2n+4+\dfrac{4}{n-1}
\\
\mathrm{DIII}(n) &
n(n-1) &

2n^2+2-\dfrac{4(n+1)}{n^2-3n+4}
\\
\mathrm{CI}(n) &
n(n+1) &

2n^2+2+\dfrac{4(n-1)}{n^2+3n+4}
\end{array}
\]

Because $s$ is an integer, $s^2$ must be an integer perfect square. We can systematically rule out all four cases by analyzing the remainder terms in the table above:

\begin{enumerate}
    \item $\mathrm{AI} \mbox{ and } \mathrm{AII}$ (Divisibility Constraints): For $s^2$ to be an integer, the fractional remainders dictate that $(n+2) \mid 8$ (for $\mathrm{AI}$ ) and $(n-1) \mid 4$ (for $\mathrm{AII}$). Within our admissible parameter bounds, these conditions restrict us entirely to $n=6$ (for AI) and $n \in \{3,5\}$ (for AII). Testing these specific values yields $s^2 \in \{33, 48, 115\}$, none of which are perfect squares.
    \item $\mathrm{DIII} \mbox{ and } \mathrm{CI}$ (Fractional Bounds): For CI, the remainder satisfies $0 < \frac{4(n-1)}{n^2+3n+4} < 1$ for all $n \ge 3$, meaning $s^2 \notin \mathbb{Z}$ which leads to a contradiction. Similarly, for DIII, the remainder is strictly bounded by $0 < \frac{4(n+1)}{n^2-3n+4} < 1$ for all $n \ge 8$, so \(s^2\notin \mathbb Z\). For \(3\le n\le 7\), a direct check shows that the only integral values occur at \(n=3\) and \(n=7\), and only \(n=3\) gives a square. But \(n=3\) implies $M=6 \implies 2pq = 6 \implies pq=3$, which contradicts the constraint $p, q \ge 2$.
\end{enumerate}

Thus no collision between AIII and any of AI, AII, DIII, or CI can occur.

\end{proof}

Now we check that there is no collision between BDI and the four one-parameter families

\begin{prop} \label{prop: collision check 3}
Let $p\ge q\ge 5$. Then the pair $(\sigma_1,\sigma_2)$ for the $\mathrm{BDI}$ family does not coincide with that of any of the four one-parameter families $\mathrm{AI}, \mathrm{AII}, \mathrm{DIII}, \mathrm{CI}$.
\end{prop}

\begin{proof}
Let $s = p+q$ and $M = pq$. For BDI, an invariant collision with a space $Y$ from one of the four one parameter families implies $\sigma_1(Y) = M$, and equating the invariant ratios yields $\frac{\sigma_2(Y)}{\sigma_1(Y)} = \frac{2M - s}{2(s-2)^2}$. We divide the four potential collisions into two analytical groups based on how the resulting Diophantine equations are constrained:

\begin{enumerate}
    \item $\mathrm{AI} \mbox{ and } \mathrm{ AII}$ (Divisibility Constraints): Equating the ratios and isolating $(s-2)^2$ yields fractional terms that strictly constrain $s$:
    \begin{enumerate}
        \item $\mathrm{AI}:$ 
$\displaystyle (s-2)^2=2n(n-1)-2s+\frac{4s}{n+2}.$ Since $s \in \mathbb{Z}$, we require $(n+2) \mid 4s$. Substituting $4s = k(n+2)$ yields the quadratic in $n$:
\[
(k^2-32)n^2+4(k^2-2k+8)n+4(k-4)^2=0.
\]
Thinking of the above equation as a quadratic in \(n\) with coefficients depending polynomially on \(k\), we observe that for \(k \geq 6\), all coefficients are positive. Hence, for every \(n \geq 0\), the left-hand side is strictly positive. Therefore, the equation admits no nonnegative real roots. A direct check for \(k<6\) shows that the only integral roots occur when \(k=3\) or \(k=4\). For \(k=3\), the only integral root is \(n=2\), while for \(k=4\), the integral roots are \(n=0,4\). In particular, all integral roots satisfy \(n<5\). Hence no admissible solution exists.
\item $\mathrm{AII}:$ 
$\displaystyle (s-2)^2=4n(2n+1)-2s-\frac{2s}{n-1}.$ Here, $(n-1) \mid 2s$. Substituting $2s = k(n-1)$ produces
\[
(k^2-32)n^2-(2k^2+4k+16)n+(k+4)^2=0.
\]
Its discriminant is
\[
\Delta=16(-k^3+9k^2+72k+144).
\]
Since the cubic is strictly decreasing for \(k>6\) and is negative at \(k=15\), we have \(\Delta<0\) for all \(k\geq 15\). Hence \(\Delta\geq 0\) only for \(1\leq k\leq 14\). Checking these finite cases shows that the discriminant is a perfect square only for \(k\in\{8,12\}\). The corresponding roots are \(\frac{9}{2},1\) for \(k=8\) and \(2,\frac{8}{7}\) for \(k=12\). Hence, neither \(k=8\) nor \(k=12\) yields an integer root \(n\geq 3\).

    \end{enumerate}
    \item $\mathrm{DIII} \mbox{ and } \mathrm{ CI}$ (Discriminant Bounds): Equating the ratios for these two families yields quadratic equations in $s$, forcing their discriminants $\Delta$ to be perfect squares:
    \begin{enumerate}
        \item $\mathrm{DIII}$:  The discriminant factors as $\Delta = 4(n-1)^2 Q(x)$, where $x = n-1$ and $Q(x) = 4x^4 + x^2 + 12x - 8$. For all $x \ge 4$ (corresponding to $n \ge 5$), $Q(x)$ is strictly bounded by consecutive perfect squares: $(2x^2)^2 < Q(x) < (2x^2+1)^2$. Thus, $\Delta$ cannot be a square. Since $pq=n(n-1)$ and $p,q\ge5$, we must have $n(n-1)\ge25$, hence $n\ge6$, and in this range we have shown $\Delta$ is never a perfect square. Hence no solution exists.
        
        \item $\mathrm{CI}$: The discriminant factors as $\Delta = 4(n+1)^2 Q(x)$, where $x = n+1$ and $Q(x) = 4x^4 + x^2 - 12x - 8$. For all $x \ge 13$ (corresponding to $n \ge 12$), we again have the strict bound $(2x^2)^2 < Q(x) < (2x^2+1)^2$. Direct calculation of the remaining finite cases ($3 \le n \le 11$) yields no perfect squares.
    \end{enumerate}
\end{enumerate}

In all four cases, no admissible integer solutions exist.

\end{proof}

We now check that there is no collision between \(\mathrm{CII}\) and the four one-parameter families.
\begin{prop} \label{prop: collision check 4}
Let \(X^{\mathrm{CII}}_{p,q}=\mathrm{Sp}(p,q)/(\mathrm{Sp}(p)\times \mathrm{Sp}(q))\) with \(p\ge q\ge 2\). Then for each of the families \(\mathrm{AI},\mathrm{AII},\mathrm{DIII},\mathrm{CI}\), there is no collision of invariants:
\[
\bigl(\sigma_1(X),\sigma_2(X)\bigr)\neq \bigl(\sigma_1(X^{\mathrm{CII}}_{p,q}),\sigma_2(X^{\mathrm{CII}}_{p,q})\bigr)
\]
for every admissible parameter choice in those families.
\end{prop}

\begin{proof}

Let $s = p+q$ and $M = pq$. We utilize the normalized invariant ratio $\rho(X) = 4\frac{\sigma_2(X)}{\sigma_1(X)}$. For the CII family, substituting the invariants from Table ~\ref{tab:symmetric-spaces-sigma invariant} yields $\rho^{\mathrm{CII}} = \frac{4M+s}{(s+1)^2}$. Since $4M = 4pq \le (p+q)^2 = s^2$, we can strictly bound this ratio:$$ \rho^{\mathrm{CII}} \le \frac{s^2+s}{(s+1)^2} = \frac{s}{s+1} < 1. $$ We divide the possible collisions into three cases according to the method used to rule them out.

\begin{enumerate}
    \item $\mathrm{AI} \mbox{ and } \mathrm{CI}$ (ratio bounds): As established in the proof of Proposition ~\ref{prop: collision check 1}, the normalized ratios for both AI and CI are strictly greater than 1 ($\rho^{\mathrm{AI}} = 1 + 2/n > 1$ and $\rho^{\mathrm{CI}} = 1 + (n+3)/(n+1)^2 > 1$). Because $\rho^{\mathrm{CII}} < 1$, collisions are globally impossible. We bypass their Diophantine systems entirely.
    \item $\mathrm{AII}$ (Divisibility Constraints): The equality \(\sigma_1^{\mathrm{AII}}(n)=\sigma_1^{\mathrm{CII}}(p,q)\) gives us
$4M=(2n+1)(n-1)$. The equality of \(\sigma_2\)-invariants then simplifies to $\displaystyle (s+1)^2=n(2n+1)+s+\frac{s}{n-1}$. Since the left-hand side is an integer, \(r:=s/(n-1)\in\mathbb Z_{>0}\). Substituting $s = r(n-1)$ produces the quadratic $F_r(n) = (r^2-2)n^2 + (-2r^2+r-1)n + (r-1)^2 = 0$. For $r=1$, $F_1(n)=-n(n+2)\neq 0$. For $r=2$, $F_2'(n)=4n-7>0$ for $n\ge3$, so $F_2$ is increasing on $[3,\infty)$; since $F_2(3)=-2$ and $F_2(4)=5$, we have $F_2(n)\neq0$ for all integers $n\ge3$. For $r\ge3$, the derivative $F_r'(n)=2(r^2-2)n-2r^2+r-1$ is strictly positive for $n\ge3$, and $F_r(3)=4r^2+r-20>0$. Thus, $F_r(n)$ is strictly increasing on $[3,\infty)$ and $F_r(n)\ge F_r(3)>0$. Hence $F_r(n)$ never vanishes in the admissible range.
\item $\mathrm{DIII}$ (Discriminant Bounds):
Equating $\sigma_1$ sets $4M = n(n-1)$. The equality of the \(\sigma_2\) invariants produces a quadratic equation in $s$. Its discriminant factors as $\Delta = (n-1)^2 Q(x)$, where $x = n-1$ and $Q(x) = 4x^4 + x^2 + 12x - 8$. As shown in Proposition ~\ref{prop: collision check 3}, for $x \ge 4$ (corresponding to $n \ge 5$), $Q(x)$ is strictly bounded by consecutive perfect squares: $(2x^2)^2 < Q(x) < (2x^2+1)^2$, preventing $\Delta$ from being a square. Direct calculation of the remaining finite cases ($n \in \{3,4\}$) yields no valid parameters. 
\end{enumerate}
In every case, a collision is impossible.

\end{proof}

\begin{prop} \label{prop: collision check 5}
Within each of the families \emph{BDI}, \emph{AIII}, and \emph{CII}, the pair
\[
\bigl(\sigma_1(X),\sigma_2(X)\bigr)
\]
determines the parameters uniquely. In particular, if two spaces from the same one of these families have the same \(\sigma_1\)- and \(\sigma_2\)-invariants, then they are equal.
\end{prop}

\begin{proof}
    Let $(p,q)$ and $(p_1, q_1)$ be two valid parameter pairs for a single family such that $p \ge q$ and $p_1 \ge q_1$. Suppose they yield the same invariants $\sigma_1$ and $\sigma_2$. For all three families, Table ~\ref{tab:symmetric-spaces-sigma invariant} shows that $\sigma_1$ is a strict scalar multiple of the product $M = pq$. Therefore, the equality $\sigma_1(p,q) = \sigma_1(p_1,q_1)$ immediately forces the products to be equal: $pq = p_1 q_1 =: M$. Let $s = p+q$. Holding the product $M$ constant, we can express the second invariant strictly as a function of the sum, $\sigma_2 = F_M(s)$. We compute this function and its monotonicity for the admissible domains in each family:

\smallskip
\noindent\textbf{AIII} ($M \geq 4, s\geq 4$): 

$$F_M(s) = \frac{M(2M+2)}{s^2}.$$
Since $M$ is a positive constant, this function is strictly decreasing with respect to $s > 0$.

\smallskip
\noindent\textbf{BDI} ($M \geq 25, s\geq 10$):  $$F_M(s) = \frac{M(2M-s)}{2(s-2)^2} \implies F_M'(s) = \frac{M(s+2-4M)}{2(s-2)^3}.$$ 
Since $p, q \ge 5$, we have $pq \ge 5p > p$ and $pq \ge 5q > q$. Summing these yields $2M > p+q = s$. Therefore, $s - 4M < -2M$. Since $M \ge 25$, the numerator factor $(s + 2 - 4M)$ is strictly negative, ensuring $F_M'(s) < 0$.

\smallskip
\noindent\textbf{CII} ($M \geq 4, s\geq 4$):
$$F_M(s) = \frac{M(4M+s)}{(s+1)^2} \implies F_M'(s) = \frac{M(1-s-8M)}{(s+1)^3}.$$
Since $M$ and $s$ are both strictly positive (with $M \ge 4$), the term $(1 - s - 8M)$ is negative. Thus, the numerator is strictly negative, ensuring $F_M'(s) < 0$.

In all three cases, $F_M(s)$ is strictly decreasing. Therefore, the equality $F_M(p+q) = F_M(p_1+q_1)$ uniquely determines the sum: $p+q = p_1+q_1 = s$. Because the pairs $(p,q)$ and $(p_1,q_1)$ share the exact same sum $s$ and product $M$, they are the roots of the same quadratic equation $z^2 - sz + M = 0$. Given our ordering convention $p \ge q$ and $p_1 \ge q_1$, it follows that $p = p_1$ and $q = q_1$.
\end{proof}
Having ruled out collisions involving the one-parameter families, as
well as collisions within each individual two-parameter family, the
only remaining possibilities are the mixed-family collisions among $\mathrm{AIII}$, $\mathrm{BDI}$, and $\mathrm{CII}$.

 \bibliographystyle{amsplain} 
\bibliography{references}    

\end{document}